\documentclass[12pt,reqno]{amsart}
\usepackage{amsmath} 

\usepackage{amsthm}
\usepackage{amssymb}
\usepackage{graphics}
\usepackage{latexsym}

\numberwithin{equation}{section}
\newtheorem{thm}{Theorem}[section]
\newtheorem{prop}[thm]{Proposition}
\newtheorem{lem}[thm]{Lemma}
\newtheorem{cor}[thm]{Corollary}

\theoremstyle{remark}
\newtheorem{rem}{Remark}[section]
\newtheorem{defn}{Definition}
 
\newcommand{\laplacian}{\Delta}
\newcommand{\BBB}{\mathbb}
\newcommand{\R}{{\BBB R}}
\newcommand{\Z}{{\BBB Z}}

\newcommand{\C}{{\BBB C}}
\newcommand{\LR}[1]{{\langle {#1} \rangle }}
\newcommand{\lec}{{\ \lesssim \ }}
\newcommand{\gec}{{\ \gtrsim \ }}

\newcommand{\cross}{\times}

\newcommand{\e}{\varepsilon}
\newcommand{\ta}{\tau}

\newcommand{\p}{\partial}
\newcommand{\la}{\lambda}

\renewcommand{\th}{\theta}
\newcommand{\de}{\delta}
\newcommand{\om}{\omega}

\newcommand{\na}{\nabla}

\newcommand{\supp}{\operatorname{supp}}

\newcommand{\I}{\infty}

\newcommand{\EQ}[1]{\begin{equation} \begin{split} #1
 \end{split} \end{equation}}
\newcommand{\EQS}[1]{\begin{align} #1 \end{align}}
\newcommand{\EQQS}[1]{\begin{align*} #1 \end{align*}}
\newcommand{\EQQ}[1]{\begin{equation*} \begin{split} #1
 \end{split} \end{equation*}}

\newcommand{\ol}{\overline}
\newcommand{\ds}{\displaystyle}
\newcommand{\sub}{\subset}

\newcommand{\F}{\mathcal{F}}
\newcommand{\1}{{\mathbf 1}}

\newcommand{\ha}{\widehat}
\newcommand{\til}{\tilde}

\title[Well-posedness of the Klein-Gordon-Zakharov system]{Well-posedness\\
for the Cauchy problem of the Klein-Gordon-Zakharov system \\
in four and more spatial dimensions
}

\author[I.  Kato]{Isao Kato}
\address[Isao Kato]{Graduate School of Mathematics, Nagoya University,
Chikusa-ku, Nagoya, 464-8602, Japan}
\email[Isao Kato]{kato.isao@f.mbox.nagoya-u.ac.jp}


\begin{document}

\begin{abstract}
We study the Cauchy problem for the Klein-Gordon-Zakharov system in spatial dimension $d \ge 4$ with 
radial or non-radial initial datum 
$(u, \p_t u, n, \p_t n)|_{t=0}$ \\
$\in H^{s+1}(\R^d) \cross H^s(\R^d) \cross \dot{H}^s(\R^d) \cross \dot{H}^{s-1}(\R^d)$.  
The critical value of $s$ is $s=s_c=d/2-2$. 
If the initial datum is radial, then we prove the small data global well-posedness and scattering at the critical space 
in $d \ge 4$ by applying the radial Strichartz estimates and $U^2, V^2$ type spaces. 
On the other hand, if the initial datum is non-radial, 
then we prove the local well-posedness at $s=1/4$ when $d=4$ and $s=s_c+1/(d+1)$ when $d \ge 5$ 
by applying the $U^2, V^2$ type spaces.  
\end{abstract}
\maketitle
\setcounter{page}{001}

%%%%%%%%%%%%%%%%%%%%%%%%%%%%%%%%%%%%%%%%%%%%%%%%%%%%%%%%%%%%%%%%%%%%%%%%%%%%%%%%%
%%%%%%%%%%%%%%%%%%%%%%%%%%%%%%%%%%%%%%%%%%%%%%%%%%%%%%%%%%%%%%%%%%%%%%%%%%%%%%%%%
%%%%%%%%%%%%%%%%%%%%%%%%%%%%%%%%%  Section 1  %%%%%%%%%%%%%%%%%%%%%%%%%%%%%%%%%%%%%
%%%%%%%%%%%%%%%%%%%%%%%%%%%%%%%%%%%%%%%%%%%%%%%%%%%%%%%%%%%%%%%%%%%%%%%%%%%%%%%%%
%%%%%%%%%%%%%%%%%%%%%%%%%%%%%%%%%%%%%%%%%%%%%%%%%%%%%%%%%%%%%%%%%%%%%%%%%%%%%%%%%

\section{Introduction}
We consider the Cauchy problem of the Klein-Gordon-Zakharov system: 
\EQS{
 \begin{cases}
  (\p_t^2  - \laplacian  + 1)u = -nu, \qquad (t,x) \in [-T,T] \cross \R^d, \\
  (\p_t^2 - c^2 \laplacian )n = \laplacian |u|^2, \qquad (t,x) \in [-T,T] \cross \R^d, \\
  (u, \p_t u, n, \p_t n)|_{t=0} = (u_0, u_1, n_0, n_1) \\
   \qquad  \qquad  \qquad  \qquad \in H^{s+1}(\R^d) \cross H^s(\R^d) \cross \dot{H}^s(\R^d) \cross \dot{H}^{s-1}(\R^d), 
                                                                                                                                            \label{KGZ}
 \end{cases}  
}
where $u, n$ are real valued functions, $d \ge 4, c > 0$ and $c \neq 1$. 
The physical model of \eqref{KGZ} is the interaction of the Langmuir wave and the ion acoustic wave in a plasma. 
In the physical model, $c$ satisfies $0 < c < 1$. 
When $d=3$, Ozawa, Tsutaya and Tsutsumi ~\cite{OTT2}  proved that \eqref{KGZ} is globally  well-posed in the energy space $H^1(\R^3) \cross L^2(\R^3) \cross L^2(\R^3) \cross \dot{H}^{-1}(\R^3)$.
They applied the Fourier restriction norm method to obtain the local well-posedness.
Then by the local well-posedness and the energy method, they obtained the global well-posedness.  
For $d=3$, Guo, Nakanishi and Wang ~\cite{GNW} proved scattering in the energy class with small, radial initial data. 
They applied the normal form reduction and the radial Strichartz estimates.     
If we transform $u_{\pm} := \om_1 u \pm i\p_t u, n_{\pm} := n \pm i(c\om)^{-1}\p_t n, 
\om_1:=(1-\laplacian )^{1/2}, \om := (-\laplacian )^{1/2}$, then \eqref{KGZ} is equivalent to the following. 
\EQS{
 \begin{cases}
  (i\p_t \mp \om_1) u_{\pm} 
    = \pm (1/4)(n_+ + n_-)(\om_1^{-1}u_+ + \om_1^{-1}u_-), 
               \quad (t,x) \in [-T,T] \cross \R^d, \\
  (i\p_t \mp c\om )n_{\pm} 
    = \pm (4c)^{-1}\om | \om_1^{-1} u_+ + \om_1^{-1} u_-|^2, \qquad (t,x) \in [-T,T] \cross \R^d, \\
  (u_{\pm}, n_{\pm})|_{t=0} = (u_{\pm 0}, n_{\pm 0}) 
                                \in H^s(\R^d) \cross \dot{H}^s(\R^d). 
                                                                                                                                            \label{KGZ'}
 \end{cases}  
}
Our main result is as follows. 
\begin{thm}  \label{mth}
 (i) Let $d=4$. Then \eqref{KGZ'} is locally well-posed in 
      $H^{1/4}(\R^4) \cross \dot{H}^{1/4}(\R^4)$. \\
 (ii) Let $d \ge 5$ and $s=(d^2-3d-2)/2(d+1)$. 
      Then \eqref{KGZ'} is locally well-posed in 
        $H^s(\R^d) \cross \dot{H}^s(\R^d)$. \\
 (iii) Let $d \ge 4, s=s_c = d/2-2$ and assume the initial data 
        $(u_{\pm 0},n_{\pm 0})$ is radial. Then, \eqref{KGZ'} is globally well-posed in 
          $H^s(\R^d) \cross \dot{H}^s(\R^d)$.   
\end{thm}
\begin{cor} \label{Scatter}
 The solution obtained in Theorem \ref{mth} (iii) scatters as $t \to \pm \I$. 
\end{cor}
For more precise statement of Theorem \ref{mth} and Corollary \ref{Scatter}, see Propositions \ref{main_prop1}, \ref{main_prop2}. 
The scaling regularity of \eqref{KGZ'} is $s=s_c=d/2-2$. 
We consider both the radial case and the non-radial case. 
First, we consider the radial case. 
In the radial case, the Strichartz estimates hold for a more wider range of $(q,r)$. 
More precisely, see Propostions \ref{rStrw}, \ref{rStrkg1}.   
On the other hand, we have to recover a half derivative loss to derive the key bilinear estimates at the critial space. 
Thanks to $c > 0$ and $c \neq 1$, if $|\xi| \gg |\xi'|$, then it holds that  
\EQS{
 M' := \max\bigl{\{} \bigl|\ta \pm c|\xi| \bigr|, |\ta' \pm \LR{\xi'}|, |\ta - \ta' \pm \LR{\xi - \xi'}| \bigr{\}} \gec |\xi|.  \label{rec}
}  
Here, $\xi, \xi'$ denote frequency for the wave equation, Klein-Gordon equation respectively and $\ta \pm c|\xi|$ (resp. $\ta' \pm \LR{\xi'}, \ta - \ta' \pm \LR{\xi - \xi'}$) denote the symbol of the linear part for the wave equation (resp. the Klein-Gordon equation).
From \eqref{rec} and by applying the $U^2, V^2$ type spaces, then we can recover the derivative loss. 
Therefore, we can obtain the bilinear estimates at the critical space by applying the radial Strichartz estimates and $U^2, V^2$ type spaces.  
Next, we consider $d=4$ and the non-radial case. 
When $d \le 4$, the Lorentz regularity $s_l$ is an important index as well as the scaling regularity for the well-posedness 
for the wave equation.  
When $d=4$ with quadratic nonlinearity, the Lorentz regularity $s_l = 1/4$.   
On the other hand, $s_c = 0$, so we need to consider $s \ge s_l = 1/4$.  
When $d=4$, we obtain local well-posedness at $s=s_l=1/4$ by applying $U^2, V^2$ type spaces.   
Finally, we consider $d \ge 5$ and the non-radial case. 
Since $s_c \ge s_l$ when $d \ge 5$, we expect the local well-posedness with $s=s_c$. 
However, we only obtain the local well-posedness with $s=s_c+1/(d+1)$. 
It seems difficult to prove the bilinear estimate with $s=s_c$. 
The reason is as below.     
We observe the first equation of \eqref{KGZ'}. 
We regard the nonlinearity as $n_{\pm}(\om_1^{-1}u_{\pm})$. 
%In the Fourier space, $n_{\pm}(\om_1^{-1}u_{\pm})$ is rewritten as $\til{n}_{\pm}(\ta-\ta', \xi-\xi')
% \LR{\xi'}^{-1}\til{u}_{\pm}(\ta', \xi')$. 
Here, we consider the following cases. 
The case $|\xi| \lec |\xi'|$ and the case $|\xi| \gg |\xi'|$, where $\xi, \xi'$ denote the frequency of $n_{\pm}, u_{\pm}$ respectively.  
For the case $|\xi| \lec |\xi'|$, the nonlinearity does not have the derivative loss,  
so we can derive the bilinear estimate at the critical space only by applying the Strichartz estimates.   
However, for the case $|\xi| \gg |\xi'|$, we need to recover a half derivative loss by \eqref{rec}.  
Here, there are three cases in \eqref{rec}. 
The cases 
$(a)\, M' = \bigl|\ta \pm c|\xi| \bigr|$, 
$(b)\, M' = |\ta' \pm \LR{\xi'}|$ and   
$(c)\, M' = |\ta-\ta' \pm \LR{\xi - \xi'}|$. 
For the case $(a)$ or $(c)$, we apply \eqref{rec} for $n_{\pm}$ and the Strichartz estimates for $\om_1^{-1}u_{\pm}$. 
Then we can obtain the bilinear estimate at the critical space. 
Whereas for $(b)$, we apply the Strichartz estimates for $n_{\pm}$ and apply \eqref{rec} for $\om_1^{-1}u_{\pm}$. 
In this case, we cannot prove the bilinear estimate at the critical space. 
As a result, we have to impose more regularity.

In section $2$, we prepare some notations and lemmas with respect to $U^p, V^p$, 
in section $3$, we prove the bilinear estimates 
and in section $4$, we prove the main result.  

\section*{Acknowledgement}
The author appreciate Professor K. Tsugawa for giving useful advice. 
Also, the author would like to thank S. Kinoshita for telling the author about the radial Strichartz estimates for the wave equation.

\section{Notations and Preliminary Lemmas}
In this section, we prepare some lemmas, propositions and notations to prove the main theorem. 
$A\lec B$ means that there exists $C>0$ such that $A \le CB.$ 
Also, $A\sim B$ means $A\lec B$ and $B\lec A.$  
Let $u=u(t,x).\ \F_t u,\ \F_x u$ denote the Fourier transform of $u$ in time, space, respectively. 
$\F_{t,\, x} u = \ha{u}$ denotes the Fourier transform of $u$ in space and time.    
%%%%%%%%%%%%%%%%%%%%%%%%%%%%%%%%
%%%%% Definition of U^p space %%%%%
%%%%%%%%%%%%%%%%%%%%%%%%%%%%%%%%
Let $\mathcal{Z}$ be the set of finite partitions $-\I<t_0<t_1<\cdots <t_K = \I$ and let $\mathcal{Z}_0$ 
be the set of finite partitions $-\I<t_0<t_1<\cdots <t_K \le \I$. 
\begin{defn}
Let $1\le p< \I.$ For $\{t_k\}_{k=0}^K \in \mathcal{Z}$ and $\{ \phi_k\}_{k=0}^{K-1}\subset L^2_x$ with   
$\sum_{k=0}^{K-1} \|\phi_k \|_{L^2_x}^p=1$, we call the function $a : \R \to L^2_x$ given by 
\EQQS{
 a=\sum_{k=1}^K \1_{[t_{k-1},\, t_k)}\phi_{k-1} 
}
a $U^p$-atom. Furthermore, we define the atomic space 
\EQQS{
 U^p:=\biggl{\{} u=\sum_{j=1}^{\I}\la_j a_j \, \Bigl| \, a_j : U^p \text{-atom} , \la_j \in \C \ such\ that\ \sum_{j=1}^{\I}|
          \la_j|< \I \biggr{\}}
}
with norm 
\EQQS{
 \| u\|_{U^p}:=\inf \biggl{\{} \sum_{j=1}^{\I}|\la_j| \, \Bigl| \, u=\sum_{j=1}^{\I}\la_j a_j, \la_j\in \C, a_j : U^p \text{-atom}
                   \biggr{\}}.
} 
\end{defn} 
%%%%%%%%%%%%%%%%%%%%%%%%%%%%%%%%
%%%%% Properties of U^p space %%%%%
%%%%%%%%%%%%%%%%%%%%%%%%%%%%%%%%
\begin{prop}
Let $1\le p<q<\I.$ \\
(i) $U^p$ is a Banach space. \\
(ii) The embeddings $U^p\subset U^q\subset L^{\I}_t(\R;L^2_x)$ are continuous. \\
(iii) For $u\in U^p$, it holds that $\lim_{t\to t_{0}+}\|u(t)-u(t_{0})\|_{L^2_x}=0,$ i.e. every $u\in U^p$ is right-continuous. \\
(iv) The closed subspace $U^p_c$ of all continuous functions in $U^p$ is a Banach space.
\end{prop}
The above proposition is in ~\cite{HHK} (Proposition 2.2).  
%%%%%%%%%%%%%%%%%%%%%%%%%%%%%%%
%%%%% Definition of V^p space %%%%%
%%%%%%%%%%%%%%%%%%%%%%%%%%%%%%%
\begin{defn}
Let $1\le p<\I.$ We define $V^p$ as the normed space of all functions $v:\R\to L^2_x$ such that 
$\lim_{t\to \pm \I}v(t)$ exist and for which the norm 
\EQQS{
\| v\|_{V^p}:=\sup_{\{ t_k\}_{k=0}^K\in \mathcal{Z}}\Bigl(\sum_{k=1}^K\| v(t_k)-v(t_{k-1})\|_{L^2_x}^p\Bigr)^{1/p}  
} 
is finite, where we use the convention that $v(-\I):=\lim_{t\to -\I}v(t)$ and $v(\I):=0.$ 
Likewise, let $V_-^p$ denote the closed subspace of all $v \in V^p$ with $\lim_{t\to -\I}v(t)=0.$ 
\end{defn}
The definitions of $V^p$ and $V^p_-$, see the erratum ~\cite{HHK2}.
%%%%%%%%%%%%%%%%%%%%%%%%%%%%%%%%
%%%%% Properties of V^p space %%%%%
%%%%%%%%%%%%%%%%%%%%%%%%%%%%%%%%
\begin{prop} \label{embedding}
Let $1\le p<q<\I.$\\
(i) Let $v:\R \to L^2_x$ be such that 
\EQQS{
 \| v\|_{V^p_0}:= \sup_{ \{ t_k\}_{k=0}^K\in \mathcal{Z}_0}\Bigl( \sum_{k=1}^K\| v(t_{k})-v(t_{k-1})\|_{L^2_x}^p\Bigr)^{1/p}
}
is finite. Then, it follows that $v(t_0^+):=\lim_{t\to t_0+} v(t)$ exists for all $t_0\in [-\I,\I)$ and 
$v(t_0^{-}):=\lim_{t\to t_0-} v(t)$ exists for all $t_0\in (-\I,\I]$ and moreover, 
\EQQS{
 \| v\|_{V^p}=\| v\|_{V^p_0}.
}
(ii) We define the closed subspace $V^p_{rc}\, (V^p_{-,\, rc})$ of all right-continuous $V^p$ functions ($V^p_-$ functions). 
The spaces $V^p,\ V^p_{rc},\ V^p_-$ and $V^p_{-,\, rc}$ are Banach spaces. \\
(iii) The embeddings $U^p\subset V^p_{-,\, rc}\subset U^q$ are continuous. \\
(iv) The embeddings $V^p\subset V^q$ and $V^p_-\subset V^q_-$ are continuous. 
\end{prop}
The proof of Proposition \ref{embedding} is in ~\cite{HHK} (Proposition 2.4 and Corollary 2.6). 
%%%%%%%%%%%%%%%%%%%%%%%%%%%%%%%%%%%%
%%%%%  Notation of operators Q_i^A  %%%%%
%%%%%%%%%%%%%%%%%%%%%%%%%%%%%%%%%%%%
Let $\{ \F_{\xi}^{-1}[\varphi_n](x)\}_{n\in \Z}\subset \mathcal{S}(\R^d)$ be the Littlewood-Paley decomposition with respect 
to $x$, that is to say
\EQQS{
\begin{cases}
 \varphi(\xi) \ge 0, \\
 \supp \varphi(\xi) = \{ \xi \,|\, 2^{-1} \le |\xi| \le 2\}, 
\end{cases}
}
\EQQS{
 \varphi_{n}(\xi) := \varphi(2^{-n}\xi),\ \sum_{n=-\I}^{\I}\varphi_{n}(\xi)=1\ (\, \xi \neq 0),\ 
 \psi(\xi) := 1-\sum_{n=0}^{\I}\varphi_{n}(\xi).
} 
Let $N=2^n\ (n \in \Z)$ be dyadic number. $P_N$ and $P_{<1}$ denote 
\EQQS{
 &\F_{x}[P_N f](\xi):=\varphi(\xi/N)\F_x[f](\xi)=\varphi_n(\xi)\F_x[f](\xi), \\
 &\F_{x}[P_{<1} f](\xi):=\psi(\xi)\F_x[f](\xi).
} 
Let $P_{\ge 1} := \sum_{N \ge 1} P_N$. 
Similarly, let $\til{Q}_N$ be 
\EQQS{
 \F_t [\til{Q}_N g](\ta):=\phi(\ta/N)\F_t[g](\ta)=\phi_n(\ta)\F_t[g](\ta),
}
where  $\{ \F_{\ta}^{-1}[\phi_{n}](t)\}_{n\in \Z}\subset \mathcal{S}(\R)$ be the Littlewood-Paley decomposition with respect 
to $t$.

Let $K_{\pm}(t)=\exp \{\mp it (1-\laplacian)^{1/2} \}: L^2_x\to L^2_x$ be the Klein-Gordon unitary operator such that 
$\F_x[K_{\pm}(t)u_0](\xi ) = \exp \{\mp it \LR{\xi} \}\, \F_x[u_0](\xi ).$ 
Similarly, we define the wave unitary operator $W_{\pm c}(t)=\exp \{\mp ict (-\laplacian)^{1/2}\}:L^2_x\to L^2_x$ such that 
$\F_x[W_{\pm c}(t)n_0](\xi ) = \exp \{\mp ict |\xi |\}\, \F_x[n_0](\xi ).$
\begin{defn}
We define \\ 
 $(i)\, U^p_{K_{\pm}}=K_{\pm}(\cdot)U^p$ with norm $\| u\|_{U^p_{K_{\pm}}}=\|K_{\pm}(-\cdot )u\|_{U^p},$\\
 $(ii)\, V^p_{K_{\pm}}=K_{\pm}(\cdot)V^p$ with norm $\| u\|_{V^p_{K_{\pm}}}=\|K_{\pm}(-\cdot )u\|_{V^p}.$\\
For dyadic number $N, M$, 
\EQQS{
 Q_N := K_{\pm}(\cdot)\til{Q}_N K_{\pm}(-\cdot), \quad 
 Q_{\ge M} := \sum_{N \ge M} Q_N, \quad 
 Q_{<M} := Id - Q_{\ge M}.
}
Here summation over $N$ means that summation over $n\in \Z$. 
Similarly, we define $U^p_{W_{\pm c}}, V^p_{W_{\pm c}}$. 
\end{defn}

\begin{rem} \label{embed}
For $L^2_x$ unitary operator $A=K_{\pm}$ or $W_{\pm c},$ 
\EQQS{
 U^2_A \subset V^2_{-,\, rc,\, A} \subset L^{\I}(\R; L^2_x)
}
\end{rem}
\begin{defn} \label{defX} 
For the Klein-Gordon equation, we define $\dot{Y}^s_{K_{\pm}}, \dot{Z}^s_{K_{\pm}}, Y^s_{K_{\pm}}, Z^s_{K_{\pm}}$ as 
the closure of all $u \in C(\R; H^s_x(\R^d))$ such that   
\EQQS{
  &\|u\|_{\dot{Y}^s_{K_{\pm}}} := \Bigl(\sum_N N^{2s}\|P_N u\|^2_{V^2_{K_{\pm}}}\Bigr)^{1/2}, 
    \quad 
     \|u\|_{\dot{Z}^s_{K_{\pm}}} := \Bigl(\sum_N N^{2s}\|P_N u\|^2_{U^2_{K_{\pm}}}\Bigr)^{1/2}, \\
  &\|u\|_{Y^s_{K_{\pm}}} := \|P_{<1}u\|_{V^2_{K_{\pm}}} + \Bigl(\sum_{N \ge 1} N^{2s}\|P_N u\|^2_{V^2_{K_{\pm}}}\Bigr)^{1/2}, \\
  &\|u\|_{Z^s_{K_{\pm}}} := \|P_{<1}u\|_{U^2_{K_{\pm}}} + \Bigl(\sum_{N \ge 1} N^{2s}\|P_N u\|^2_{U^2_{K_{\pm}}}\Bigr)^{1/2}. 
} 
Similarly, for the wave equation, we define 
$\dot{Y}^s_{W_{\pm c}}, \dot{Z}^s_{W_{\pm c}}, Y^s_{W_{\pm c}}, Z^s_{W_{\pm c}}$ by replacing 
$K_{\pm}$ with $W_{\pm c}$ in the above norms.    
\end{defn}
\begin{defn}
For a Hilbert space $H$ and a Banach space $X\subset C(\R; H)$, we define
\EQQS{
 &B_r(H):=\{ f \in H\, |\, \|f\|_H \le r \}, \\
 &X([0,T)):=\{ u \in C([0,T);H)\, |\, ^{\exists} \til{u} \in X , \til{u}(t)=u(t), t \in [0,T) \} 
}
endowed with the norm $\|u\|_{X([0,T))}=\inf \{ \|\til{u}\|_X |\,  \til{u}(t)=u(t), t \in [0,T)\}.$
\end{defn}
We denote the Duhamel term 
\EQQS{
 &I_{T, K_{\pm}}(n,v) := \pm \int_0^t \1_{[0,T]}(t')K_{\pm}(t-t')n(t')(\om_1^{-1}v(t'))dt', \\
 &I_{T, W_{\pm c}}(u,v) := \pm \int_0^t \1_{[0,T]}(t')W_{\pm c}(t-t')\om \bigl( (\om_1^{-1}u(t')) (\ol{\om_1^{-1}v(t')})\bigr) dt'
}
for the Klein-Gordon equation and the wave equation respectively. 

\begin{lem} \label{Recovery}
Let $c > 0, c \neq 1$ and $\ta_3=\ta_1-\ta_2,\ \xi_3=\xi_1-\xi_2.$ 
If $|\xi_1| \gg |\xi_2|$ or $|\xi_1| \ll |\xi_2|,$ then it holds that 
\EQS{ \label{recover}
 \max\big\{ \big|\ta_1 \pm |\xi_1| \big|, \big|\ta_2 \pm |\xi_2| \big|, \big|\ta_3 \pm c|\xi_3| \big|  \big\} 
  \gec \max\{|\xi_1|, |\xi_2|\}.
} 
\end{lem}
\begin{proof}
We only prove the case $|\xi_1| \gg |\xi_2|$ since the case $|\xi_1| \ll |\xi_2|$ is proved by the same manner.  
\EQS{
 (\text{l.h.s}) \gec | (\ta_1 \pm |\xi_1|) - (\ta_2 \pm |\xi_2|) - (\ta_3 \pm c|\xi_3|)|     \label{recov}
}
If $0 < c < 1$, then the right hand side of \eqref{recov} is bounded by 
\EQQS{
 |\xi_1| - |\xi_2| - c|\xi_1 - \xi_2|  
   \gec |\xi_1|.
}
If $c > 1$, then the right hand side of \eqref{recov} is bounded by 
\EQQS{
 c|\xi_1 - \xi_2| - |\xi_1| - |\xi_2| 
   \gec |\xi_1|.
}   
\end{proof}
The following proposition is in ~\cite{HHK} (Theorem 2.8 and Proposition 2.10). 
\begin{prop}  \label{U2norm}
$u\in V^1_-\sub U^2$ be absolutely continuous on compact intervals and $v \in V^2.$
Then, $\|u\|_{U^2}=\ds \sup_{v \in V^2,\, \|v\|_{V^2}=1}\Bigl|\int_{-\I}^{\I}\LR{u'(t), v(t)}_{L^2_x} dt\Bigr|.$
\end{prop}
\begin{cor}  \label{U2A}
Let $A=K_{\pm}$ or $W_{\pm c}$ and $u \in V^1_{-,\, A} \subset U^2_A$ be absolutely continuous on compact intervals and 
$v \in V^2_A.$
Then, 
\EQQS{
 \|u\|_{U^2_A} = \sup_{v \in V^2_A,\, \|v\|_{V^2_A}=1}
                      \Bigl|\int_{-\I}^{\I} \LR{A(t)(A(-\cdot )u)'(t), v(t)}_{L^2_x} dt\Bigr|.
}
\end{cor}
\begin{lem}  \label{p}
Let $M >0$ and $Q \in \{ Q_{<M}, Q_{\ge M}\}$. 
For $1 \le p \le \I$ and $f \in V^2_{K_{\pm}}$, it holds that 
\EQS{ 
 \|Q (\1_{[0,T]}f)\|_{L^p_t L^2_x} \lec T^{1/p} \|f\|_{V^2_{K_{\pm}}}.         \label{p'}
}
\end{lem}
\begin{proof}
By scaling, we only prove \eqref{p'} for $M=1$. 
We will show \eqref{p'} for $Q = Q_{<1}$, since $Q_{\ge 1}=Id-Q_{<1}$.   
Put $g := K_{\pm}(-\cdot)f$. 
Then \eqref{p'} is equivalent to 
\EQS{
 \|Q_{<1}(\1_{[0,T]}) K_{\pm}(\cdot)g)\|_{L^p_t L^2_x} \lec T^{1/p}\|g\|_{V^2}.         \label{p2}
} 
By the unitarity of $K_{\pm}$, we have 
\EQS{
 \|Q_{<1}(\1_{[0,T]}K_{\pm}(\cdot) g)\|_{L^p_t L^2_x} 
  &= \Bigl{\|} \sum_{N < 1} K_{\pm}(\cdot) \til{Q}_N K_{\pm}(-\cdot) 
      (\1_{[0,T]}K_{\pm}(\cdot)g)\Bigr{\|}_{L^p_t L^2_x} \notag \\
  &= \Bigl{\|} \sum_{N < 1} \til{Q}_N (\1_{[0,T]}g)\Bigr{\|}_{L^p_t L^2_x}  \notag \\
  &= \|\til{Q}_{<1}(\1_{[0,T]}g)\|_{L^p_t L^2_x}.                                                  \label{p3}
}
For some Schwartz function $\phi$, it holds that 
\EQQS{
 \til{Q}_{<1}h = \phi \ast_t h. 
} 
Hence by the Young inequality and the H\"{o}lder inequality, we have 
\EQS{
 \|\til{Q}_{<1}(\1_{[0,T]}g)\|_{L^p_t L^2_x} 
   &\lec \|\phi\|_{L^1_t} \|\1_{[0,T]}g\|_{L^p_t L^2_x} \notag \\
   &\lec \|\1_{[0,T]}\|_{L^p_t} \|g\|_{L^{\I}_t L^2_x} \notag \\
   &\lec T^{1/p}\|g\|_{V^2}.                                                                        \label{p4}
}
Collecting \eqref{p3}--\eqref{p4}, we obtain \eqref{p2}.
\end{proof}

\begin{prop} \label{mlinear}
Let $T_0:\, L^2_x \cross \dots \cross L^2_x\to L^1_{loc}(\R^d;\C)\ $ 
be a n-linear operator.
Assume that for some $ 1\le p, q \le \I$, it holds that
$$\|T_0( K_{\pm}(\cdot)\phi_1, \dots ,K_{\pm}(\cdot)\phi_n)\|_{L^p_t(\R;L^q_x(\R^d))}
   \lec \ds \prod_{i=1}^n\|\phi_i\|_{L^2_x}.$$
Then, there exists $T:U_{K_{\pm}}^p \cross \dots \cross U_{K_{\pm}}^p\to 
      L^p_t(\R;L^q_x(\R^d))$ satisfying  
$$\|T(u_1, \dots , u_n)\|_{L^p_t(\R;L^q_x(\R^d))} \lec \ds \prod_{i=1}^n
                    \|u_i\|_{U^p_{K_{\pm}}},$$
such that
$T(u_1, \dots, u_n)(t)(x)=T_0(u_1(t), \dots , u_n(t))(x)$
a.e. 
\end{prop}
See Proposition 2.19 in ~\cite{HHK} for the proof of the above proposition. 
\begin{prop} \label{Strich-w}
Let $2 \le r < \I, 2/q = d(1/2-1/r), s = 1/q-1/r+1/2$. Then it holds that 
\EQQS{
 \| W_{\pm c}(t) f\|_{L^q_t \dot{W}^{-s,r}_x(\R^{1+d})} \lec \|f\|_{L^2_x(\R^d)}. 
}
\end{prop}
For the proof of Proposition \ref{Strich-w}, see ~\cite{K}, ~\cite{GV}. 
\begin{prop} \label{Strich-kg}
Let $2 \le r < \I, 2/q = d(1/2-1/r), s = 1/q-1/r+1/2$. Then, it holds that 
\EQQS{
 \| K_{\pm}(t) f\|_{L^q_t W^{-s,r}_x(\R^{1+d})} \lec \|f\|_{L^2_x(\R^d)}. 
}
\end{prop}
For the proof of Proposition \ref{Strich-kg}, see ~\cite{MNO2}. 
Combining Proposition \ref{embedding}, Proposition \ref{Strich-w}, Proposition \ref{Strich-kg} and Proposition \ref{mlinear}, 
we have the following proposition.   
\begin{prop} \label{Str}
Let $2 \le r < \I, 2/q = d(1/2-1/r), s = 1/q-1/r+1/2$.  
If $p > q$, then it holds that 
\EQQS{
 \|f\|_{L^q_t W^{-s,r}_x(\R^{1+d})} \lec \|f\|_{V^p_{K_{\pm}}}, \qquad 
 \|f\|_{L^q_t \dot{W}^{-s,r}_x(\R^{1+d})} \lec \|f\|_{V^p_{W_{\pm c}}}.
}
\end{prop}
\begin{prop} \label{rStrw}
Let $d \ge 3$. Then, for all radial functions $f \in L^2_x(\R^d)$, it holds that 
\EQS{
 \| W_{\pm c}(t) P_N f\|_{L^q_t L^r_x(\R^{1+d})} \lec N^{d(1/2-1/r)-1/q}\|f\|_{L^2_x(\R^d)},         \label{wrad}
} 
if and only if 
\EQS{
 (q,r)=(\I,2) \quad \text{or} \quad 2 \le q \le \I, \quad 1/q < (d-1)(1/2-1/r).                       \label{wradadm}
}
\end{prop}
See Theorem 1.5 $(a)$ in ~\cite{GW} for the proof of Proposition \ref{rStrw}. 
\begin{prop} \label{rStrkg1}
If $2d/(d-1) < q \le \I, N \ge 1$ and $f \in L^2_x(\R^d)$ is radial function, then it holds that 
\EQS{
 \| K_{\pm}(t) P_N f\|_{L^q_{t,x}(\R^{1+d})} \lec N^{d/2-(d+1)/q}\|f\|_{L^2_x(\R^d)}.             \label{kgradhi}
} 
If $N < 1$ and $f \in L^2_x(\R^d)$ is radial function, then it holds that 
\EQS{
 \| K_{\pm}(t) P_N f\|_{L^q_{t,x}(\R^{1+d})} \lec N^{d/2-(d+2)/q}\|f\|_{L^2_x(\R^d)}.             \label{kgradlo}
} 
\end{prop}
See (3.13) in ~\cite{GW} for the proof of Proposition \ref{rStrkg1}. 
\begin{prop} \label{rStrkg2}
Let $(q,r)$ satisfy $2 \le r \le 2(d+1)/(d-1), r \le q, (1/2)(d-1)(1/2-1/r) \le 1/q < (d-1)(1/2-1/r)$ and $N \ge 1$. 
Then, for all radial function $f \in L^2_x(\R^d)$, it holds that 
\EQS{
 \| K_{\pm}(t) P_N f\|_{L^q_t L^r_x(\R^{1+d})} \lec N^{d(1/2-1/r)-1/q}\|f\|_{L^2_x(\R^d)}.      \label{qr}
}   
\end{prop}
\begin{proof}
When $q=r$, \eqref{qr} follows from \eqref{kgradhi}. 
Interpolating $L^q_{t,x}$ with $L^{\I}_t L^2_x$, we obtain \eqref{qr}. 
\end{proof}
\begin{prop} \label{rStr}
(i) Let $d \ge 3, (q,r)$ satisfy \eqref{wradadm} in Proposition \ref{rStrw} and $s=d(1/2-1/r)-1/q$. 
If $p > q$, then it holds that 
\EQQS{
 \|u\|_{L^q_t \dot{W}^{-s,r}_x(\R^{1+d})} \lec \|u\|_{V^p_{W_{\pm c}}}. 
}
(ii) Let $(q,r)$ satisfy the condition in Proposition \ref{rStrkg2}. 
If $p > q$ and $s_1 = d(1/2-1/r)-1/q$, then it holds that 
\EQS{
 \| P_{\ge 1} u\|_{L^q_t \dot{W}^{-s_1,r}_x(\R^{1+d})} \lec \|u\|_{V^p_{K_{\pm}}}.          \label{radhi}
}
(iii) If $p > q$ and $s_2 = d/2-(d+2)/q$, then it holds that 
\EQS{
 \| P_{<1} u\|_{L^q_t \dot{W}^{-s_2,q}_x(\R^{1+d})} \lec \|u\|_{V^p_{K_{\pm}}}.              \label{radlo}
}
\end{prop}
Combining Proposition \ref{embedding}, Proposition \ref{rStrw} and Proposition \ref{mlinear}, 
we have Proposition \ref{rStr} $(i)$. 
Combining Proposition \ref{embedding}, Proposition \ref{rStrkg2} and Proposition \ref{mlinear}, 
we have Proposition \ref{rStr} $(ii)$.
Combining Proposition \ref{embedding}, \eqref{kgradlo} and Proposition \ref{mlinear}, 
we have Proposition \ref{rStr} $(iii)$.

\begin{prop} \label{unique}  
(i) Let $T>0$ and $u \in Y^s_{K_{\pm}}([0,T]), u(0)=0$. Then, there exists $0 \le T' \le T$ such that 
$\|u\|_{Y^s_{K_{\pm}}([0,T'])} < \e $.   \\
(ii) Let $T>0$ and $n \in \dot{Y}^s_{W_{\pm c}}([0,T]), n(0)=0$. 
Then, there exists $0 \le T' \le T$ such that $\|n\|_{\dot{Y}^s_{W_{\pm c}}([0,T'])} < \e$. 
\end{prop}
For the proofs of $(i)$ and $(ii)$, see Proposition 2.24 in \cite{HHK}.

\begin{lem} \label{estX}
Let $a \ge 0$. Then for $A=K_{\pm}$ or $W_{\pm c}$, it holds that 
\EQQS{
 \|\LR{\na _x}^a f\|_{V^2_A} \lec \|f\|_{Y^a_A}. 
}
\end{lem}
\begin{proof}
We only prove for $A=K_{\pm}$ since we can prove similarly for $A=W_{\pm c}$. 
By $L^2_x$ orthogonality, we have  
\EQQS{ 
  \|\LR{\na _x}^a f\|_{V^2_{K_{\pm}}}^2 
    &\lec \sup_{\{t_i\}_{i=0}^I \in \mathcal{Z}} 
              \sum_{i = 1}^I (\|P_{<1}(K_{\pm}(-t_i)f(t_i)-K_{\pm}(-t_{i-1})f(t_{i-1}))\|_{L^2_x}^2 \notag \\
    &\quad               
              + \sum_{N \ge 1} N^{2a} \|P_N(K_{\pm}(-t_i)f(t_i)-K_{\pm}(-t_{i-1})f(t_{i-1}))\|_{L^2_x}^2) \notag \\
    &\lec \sup_{\{t_i\}_{i=0}^I \in \mathcal{Z}} 
              \sum_{i=1}^I \|K_{\pm}(-t_i)P_{<1}f(t_i)-K_{\pm}(-t_{i-1})P_{<1}f(t_{i-1})\|_{L^2_x}^2 \notag \\
    &\quad    
              + \sum_{N \ge 1} N^{2a} \sup_{\{t_i\}_{i=0}^I \in \mathcal{Z}} 
                   \sum_{i=1}^I \|K_{\pm}(-t_i)P_N f(t_i)-K_{\pm}(-t_{i-1})P_N f(t_{i-1})\|_{L^2_x}^2 \notag \\
    &\lec \|f\|_{Y^a_{K_{\pm}}}^2.
}
\end{proof}
\begin{rem} \label{estx}
Similarly, we see 
\EQQS{
 \| |\na _x|^a f\|_{V^2_A} \lec \|f\|_{\dot{Y}^a_A}.  
}
\end{rem}

\begin{prop} \label{modulation}
It holds that 
\EQS{
 &\|Q_M u\|_{L^2_{t,x}(\R^{1+d})} \lec M^{-1/2}\|u\|_{V^2_{K_{\pm}}},\ \ \ \ \ 
      \|Q_{\ge M} u\|_{L^2_{t,x}(\R^{1+d})} \lec M^{-1/2}\|u\|_{V^2_{K_{\pm}}}, \label{mod} \\
  &\|Q_{<M} u\|_{V^2_{K_{\pm}}} \lec \|u\|_{V^2_{K_{\pm}}},\ \ \ \ \ 
      \|Q_{\ge M} u\|_{V^2_{K_{\pm}}} \lec \|u\|_{V^2_{K_{\pm}}}, \notag \\
  &\|Q_{<M} u\|_{U^2_{K_{\pm}}} \lec \|u\|_{U^2_{K_{\pm}}},\ \ \ \ \ 
      \|Q_{\ge M} u\|_{U^2_{K_{\pm}}} \lec \|u\|_{U^2_{K_{\pm}}}. \notag 
}
The same estimates hold by replacing the Klein-Gordon operator $K_{\pm}$ by the wave operator $W_{\pm c}.$  
\end{prop}

\begin{lem} \label{Mhigh}
If $f, g$ are measurable functions, then 
\EQQS{
 \int_{\R^{1+d}} f(t,x) \ol{Q_{\ge M} g(t,x)} dxdt = \int_{\R^{1+d}} (Q_{\ge M} f(t,x))\ol{g(t,x)}dxdt. 
}
\end{lem}
For the proof of Lemma \ref{Mhigh}, see ~\cite{KaT}.

\begin{lem}  \label{tri}
Let $\til{u}_{N_1}:=\1_{[0,T)}P_{N_1} u, \til{v}_{N_2}:=\1_{[0,T)}P_{N_2} v, \til{n}_{N_3}:=\1_{[0,T)}P_{N_3} n, 
Q_1, Q_2 \in \{ Q_{<M},Q_{\ge M}\}, Q_3 \in \{ Q_{<M},Q_{\ge M}\}, s':=(d^2-3d-2)/2(d+1), s_c:=d/2-2$.  
Then the following estimates hold for sufficiently small $T>0$ if $\th > 0$, and hold for all $0<T<\I$ if $\th = 0$ or spherically symmetric $(u,v,n)$: \\ 
 (i) If $N_3 \lec N_2 \sim N_1$, then  
\EQQS{
 |I_1| &:= \Bigl|\int_{\R^{1+d}}(\om_1^{-1}\til{u}_{N_1})(\ol{\om_1^{-1}\til{v}_{N_2}})(\ol{\om \til{n}_{N_3}})dxdt\Bigr| \\
  &\lec T^{\th}N_3^{s}\|u_{N_1}\|_{V^2_{K_{\pm}}}\|v_{N_2}\|_{V^2_{K_{\pm}}} \|n_{N_3}\|_{V^2_{W_{\pm c}}}, 
}
when $(\th, s)=(1/4, 1/4)$ for $d=4$ and $(\th, s)=(0, s_c), (1/(d+1), s')$  for $d \ge 5$. 
Moreover, if $(u, v, n)$ are spherically symmetric, then for $d \ge 4$, 
\EQQS{
 &|I_1| \lec \LR{N_2}^{(d-8)/3}N_3^{(d+4)/6}\|u_{N_1}\|_{V^2_{K_{\pm}}}\|v_{N_2}\|_{V^2_{K_{\pm}}}\|n_{N_3}\|_{V^2_{W_{\pm c}}}. 
}
 (ii) $|I_2| := \bigl|\int_{\R^{1+d}} \til{n} (\om_1^{-1}\til{v})(\ol{P_{<1}\til{u}})dxdt\bigr|
                       \lec T^{\th} \|n\|_{\dot{Y}^s_{W_{\pm c}}}\|v\|_{Y^s_{K_{\pm}}}\|P_{<1}u\|_{V^2_{K_{\pm}}}$, \\
when $(\th, s) = (1, 1/4)$ for $d=4$ and $(\th, s) = (0, s_c), (1/(d+1), s')$ for $d \ge 5$. 
Moreover, if $(u,v,n)$ are spherically symmetric, then for $4 \le d <5$, 
\EQQS{
 &|I_2| \lec \|n\|_{\dot{Y}^{s_c}_{W_{\pm c}}}\|v\|_{Y^{s_c}_{K_{\pm}}}\|P_{<1}u\|_{V^2_{K_{\pm}}}. 
}
 (iii) If $N_3 \lec N_2 \sim N_1$, then  
\EQQS{
  |I_3| := \Bigl|\int_{\R^{1+d}}\Bigl(\sum_{N_3 \lec N_2}\til{n}_{N_3}\Bigr)(\om_1^{-1}\til{v}_{N_2})\ol{\til{u}_{N_1}}dxdt\Bigr|
                   \lec T^{\th}\|n\|_{\dot{Y}^s_{W_{\pm c}}}\|v_{N_2}\|_{V^2_{K_{\pm}}}\|u_{N_1}\|_{V^2_{K_{\pm}}}, 
}
when $(\th, s) = (1/4, 1/4)$ for $d=4$ and $(\th, s) = (0, s_c), (1/(d+1), s')$ for $d \ge 5$. 
Moreover, if $(u,v,n)$ are spherically symmetric, then for $4 \le d <5$, 
\EQQS{ 
   |I_3| \lec \|n\|_{\dot{Y}^{s_c}_{W_{\pm c}}}\|v_{N_2}\|_{V^2_{K_{\pm}}}\|P_{<1}u\|_{V^2_{K_{\pm}}}. 
}
 (iv) If $N_2 \ll N_1 \sim N_3, N_1 \ge 1, M=\e N_1$ and $\e > 0$ is sufficiently small, then   
\EQQS{
   |I_i| \lec T^{\th}\|n_{N_3}\|_{V^2_{W_{\pm c}}} \|v\|_{Y^s_{K_{\pm}}} \|u_{N_1}\|_{V^2_{K_{\pm}}}, 
} 
when $(\th, s) = (1/4, 1/4)$ for $d=4, i=4,5,6$ and $(\th, s) = (0, s_c)$ for $d \ge 4, i=4,6$ 
and $(\th, s) = (1/(d+1), s')$ for $d \ge 5, i=4,5,6$. 
Moreover, if $(u,v,n)$ are spherically symmetric, then for $d \ge 4$, it holds that 
\EQQS{ 
 |I_5| \lec \|n_{N_3}\|_{V^2_{W_{\pm c}}} \|v\|_{Y^{s_c}_{K_{\pm}}} \|u_{N_1}\|_{V^2_{K_{\pm}}},  
}
where 
\EQQS{ 
 &I_4 := \int_{\R^{1+d}}(Q_{\ge M}\til{n}_{N_3})\Bigl(\sum_{N_2 \ll N_1}Q_2 \om_1^{-1}\til{v}_{N_2}\Bigr)
                      (\ol{Q_1 \til{u}_{N_1}})dxdt, \\
 &I_5 := \int_{\R^{1+d}}(Q_3 \til{n}_{N_3})\Bigl(\sum_{N_2 \ll N_1}Q_{\ge M} \om_1^{-1}\til{v}_{N_2}\Bigr)
                      (\ol{Q_1 \til{u}_{N_1}})dxdt, \\
 &I_6 := \int_{\R^{1+d}}(Q_3 \til{n}_{N_3})\Bigl(\sum_{N_2 \ll N_1}Q_2 \om_1^{-1}\til{v}_{N_2}\Bigr)
                      (\ol{Q_{\ge M} \til{u}_{N_1}})dxdt.
}
\end{lem}
\begin{proof}
We show $(i)$ first.  
For $f \in V^2_A,\ A \in \{ K_{\pm}, W_{\pm c} \}$, we see 
\EQS{  \label{c}
 \|\1_{[0,T)}f\|_{V^2_A} \lec \|f\|_{V^2_A}.
}
First, we show $d=4$. 
We apply the H\"{o}lder inequality, Proposition \ref{Str}, the Sobolev inequality, \eqref{c} and $N_3 \lec N_1 \sim N_2$, 
then we have  
\EQQS{
 |I_1|   
  &\lec \|\1_{[0,T)}\|_{L^4_t}\|\om_1^{-1}\til{u}_{N_1}\|_{L^{20/3}_t L^{5/2}_x} \|\om_1^{-1}\til{v}_{N_2}\|_{L^{20/3}_t L^{5/2}_x}
                    \|\om \til{n}_{N_3}\|_{L^{20/9}_t L^5_x}\\
  &\lec T^{1/4} \| \LR{\na_x}^{1/4}\om_1^{-1}\til{u}_{N_1}\|_{V^2_{K_{\pm}}} \| \LR{\na_x}^{1/4}\om_1^{-1}
                    \til{v}_{N_2}\|_{V^2_{K_{\pm}}} \| |\na_x|^{3/4}\om \til{n}_{N_3}\|_{V^2_{W_{\pm c}}} \\
  &\lec T^{1/4} \LR{N_1}^{1/4-1}\|u_{N_1}\|_{V^2_{K_{\pm}}} \LR{N_2}^{1/4-1}\|v_{N_2}\|_{V^2_{K_{\pm}}} N_3^{3/4+1}
                    \|n_{N_3}\|_{V^2_{W_{\pm c}}} \\
  &\lec T^{1/4} N_3^{1/4}\|u_{N_1}\|_{V^2_{K_{\pm}}}\|v_{N_2}\|_{V^2_{K_{\pm}}} \|n_{N_3}\|_{V^2_{W_{\pm c}}}.
}
For $d \ge 5$, we apply the H\"{o}lder inequality to have  
\EQS{
 |I_1| 
  \lec \|\om_1^{-1}\til{u}_{N_1}\|_{L^{2(d+1)/(d-1)}_{t,x}} \|\om_1^{-1}\til{v}_{N_2}\|_{L^{2(d+1)/(d-1)}_{t,x}}
                         \|\om \til{n}_{N_3}\|_{L^{(d+1)/2}_{t,x}}.                                             \label{5dn} 
}
We apply Proposition \ref{Str}, \eqref{c} and the Sobolev inequality, then we have 
\EQS{
  \|\om_1^{-1}\til{f}_N\|_{L^{2(d+1)/(d-1)}_{t,x}}
      &\lec \LR{N}^{1/2-1} \|f_N\|_{V^2_{K_{\pm}}} 
        =     \LR{N}^{-1/2} \|f_N\|_{V^2_{K_{\pm}}},                                                        \label{5dn2} \\
  \|\om \til{n}_{N_3}\|_{L^{(d+1)/2}_{t,x}} 
      &\lec \| |\na_x|^{d(d-5)/2(d-1)}\om \til{n}_{N_3}\|_{L^{(d+1)/2}_t L^{2(d^2-1)/(d^2-9)}_x}  \notag \\
      &\lec \| |\na_x|^{d/2-2}\om \til{n}_{N_3}\|_{V^2_{W_{\pm c}}}                                     \label{5dn3} \\
      &\lec N_3^{s_c + 1}\|n_{N_3}\|_{V^2_{W_{\pm c}}}                                                    \label{5dn4} 
}
Collecting \eqref{5dn}, \eqref{5dn2}, \eqref{5dn4} and $N_3 \lec N_1 \sim N_2$, we obtain 
\EQQS{
 |I_1| \lec N_3^{s_c} \|u_{N_1}\|_{V^2_{K_{\pm}}} \|v_{N_2}\|_{V^2_{K_{\pm}}} \|n_{N_3}\|_{V^2_{W_{\pm c}}}.     
}
In \eqref{5dn}, if we apply the H\"{o}lder inequality, the Sobolev inequality and Proposition \ref{Str}, then we have 
\EQS{
 \|\om \til{n}_{N_3}\|_{L^{(d+1)/2}_{t,x}} 
  &\lec \| \1_{[0,T)}\|_{L^{d+1}_t} \|\om n_{N_3}\|_{L^{d+1}_t L^{(d+1)/2}_x} \notag \\
  &\lec T^{1/(d+1)} \| |\na_x|^{d(d^2-4d-1)/2(d^2-1)}\om n_{N_3}\|_{L^{d+1}_t L^{2(d^2-1)/(d^2-5)}_x} \notag \\
  &\lec T^{1/(d+1)} \| |\na_x|^{(d^2-3d-2)/2(d+1)}\om n_{N_3}\|_{V^2_{W_{\pm c}}}                \label{5dn5} \\
  &\lec T^{1/(d+1)} N_3^{s'+1} \|n_{N_3}\|_{V^2_{W_{\pm c}}}.                                            \label{5dn6}
}
Collecting \eqref{5dn}, \eqref{5dn2}, \eqref{5dn6} and $N_3 \lec N_1 \sim N_2$, we obtain 
\EQQS{
 |I_1| \lec T^{1/(d+1)} N_3^{s'}  \|u_{N_1}\|_{V^2_{K_{\pm}}} \|v_{N_2}\|_{V^2_{K_{\pm}}} \|n_{N_3}\|_{V^2_{W_{\pm c}}}. 
} 
Next, we prove for $d \ge 4$ and $(u,v,n)$ are spherically symmetric functions. 
We apply the H\"{o}lder inequality to have 
\EQS{
 |I_1| \lec \|\om_1^{-1}\til{u}_{N_1}\|_{L^3_{t,x}} \|\om_1^{-1}\til{v}_{N_2}\|_{L^3_{t,x}} \|\om \til{n}_{N_3}\|_{L^3_{t,x}}. \label{4dr}
}
We apply Proposition \ref{rStr}, \eqref{c} and $N_3 \lec N_2 \sim N_1$, 
then we have  
\EQS{
  \|\om_1^{-1}\til{u}_{N_1}\|_{L^3_{t,x}} 
   &\lec \| |\na_x|^{(d-4)/6}P_{<1}\om_1^{-1}\til{u}_{N_1}\|_{V^2_{K_{\pm}}} 
             + \Bigl{\|} |\na_x|^{(d-2)/6} \sum_{N \ge 1} P_N (\om_1^{-1}\til{u}_{N_1})\Bigr{\|}_{V^2_{K_{\pm}}} \notag \\
   &\lec \LR{N_1}^{(d-2)/6-1}\Bigl( \|P_{<1}u_{N_1}\|_{V^2_{K_{\pm}}} 
             + \Bigl{\|} \sum_{N \ge 1} P_N u_{N_1}\Bigr{\|}_{V^2_{K_{\pm}}}\Bigr) \notag \\
   &\lec \LR{N_2}^{(d-8)/6}\|u_{N_1}\|_{V^2_{K_{\pm}}},                                                 \label{4dr2} \\
  \|\om_1^{-1}\til{v}_{N_2}\|_{L^3_{t,x}} &\lec \LR{N_2}^{(d-8)/6}\|v_{N_2}\|_{V^2_{K_{\pm}}},     \label{4dr3} \\
  \|\om \til{n}_{N_3}\|_{L^3_{t,x}} 
   &\lec \| |\na_x|^{(d-2)/6} \om \til{n}_{N_3}\|_{V^2_{W_{\pm c}}} 
     \lec N_3^{(d+4)/6}\|n_{N_3}\|_{V^2_{W_{\pm c}}}.                                                     \label{4dr4} 
}
From \eqref{4dr}--\eqref{4dr4}, we obtain 
\EQQS{
 |I_1| \lec \LR{N_2}^{(d-8)/3}N_3^{(d+4)/6}\|u_{N_1}\|_{V^2_{K_{\pm}}}\|v_{N_2}\|_{V^2_{K_{\pm}}} \|n_{N_3}\|_{V^2_{W_{\pm c}}}.
}
Next, we prove $(ii)$. 
For $d=4$, 
we apply the H\"{o}lder inequality, the Sobolev inequality, Remark \ref{embed}, \eqref{c}, Remark \ref{estx}, 
discarding $\om_1^{-1}$ and Lemma \ref{estX}, we obtain 
\EQQS{
 |I_2|
      &\lec \|\1_{[0,T)}\|_{L^1_t} \|\til{n}\|_{L^{\I}_t L^{16/7}_x}
                    \|\om_1^{-1}\til{v}\|_{L^{\I}_t L^{16/7}_x} \|P_{<1}\til{u}\|_{L^{\I}_t L^8_x} \\
      &\lec T \| |\na_x|^{1/4}\til{n}\|_{L^{\I}_t L^2_x} \|\LR{\na_x}^{1/4}\om_1^{-1}\til{v}\|_{L^{\I}_t L^2_x}
                    \| |\na_x|^{3/2}P_{<1}\til{u}\|_{L^{\I}_t L^2_x} \\
      &\lec T \| |\na_x|^{1/4}\til{n}\|_{V^2_{W_{\pm c}}} 
                    \|\LR{\na_x}^{1/4}\om_1^{-1}v\|_{V^2_{K_{\pm}}} \|P_{<1}u\|_{V^2_{K_{\pm}}} \\
      &\lec T \|n\|_{\dot{Y}^{1/4}_{W_{\pm c}}} \|v\|_{Y^{1/4}_{K_{\pm}}} \|P_{<1}u\|_{V^2_{K_{\pm}}}.                 
}
For $d \ge 5$, by the H\"{o}lder inequality to have 
\EQS{
 |I_2| \lec \|\til{n}\|_{L^{(d+1)/2}_{t,x}} \|\om_1^{-1}\til{v}\|_{L^{2(d+1)/(d-1)}_{t,x}}
                 \| P_{<1}\til{u}\|_{L^{2(d+1)/(d-1)}_{t,x}}.                                                                      \label{5di2a}
}
From \eqref{5dn2}, \eqref{5dn3}, Remark \ref{estx} and Lemma \ref{estX}, we obtain 
\EQS{ 
 &\|\til{n}\|_{L^{(d+1)/2}_{t,x}} \lec \|n\|_{\dot{Y}^{s_c}_{W_{\pm c}}},                                                     \label{5di2b} \\
 &\|\om_1^{-1}\til{v}\|_{L^{2(d+1)/(d-1)}_{t,x}} \lec \| \LR{\na_x}^{-1/2} v\|_{V^2_{K_{\pm}}} 
                                                         \lec \|\LR{\na_x}^{s_c}v\|_{V^2_{K_{\pm}}}  
                                                         \lec \|v\|_{Y^{s_c}_{K_{\pm}}},                                           \label{5di2c} \\
 &\|P_{<1}\til{u}\|_{L^{2(d+1)/(d-1)}_{t,x}} \lec \|\LR{\na_x}^{1/2} P_{<1}u\|_{V^2_{K_{\pm}}}
                                                   \lec \|P_{<1}u\|_{V^2_{K_{\pm}}}.                                              \label{5di2d} 
}
Collecting \eqref{5di2a}--\eqref{5di2d}, we obtain 
\EQQS{
 |I_2| \lec \|n\|_{\dot{Y}^{s_c}_{W_{\pm c}}} \|v\|_{Y^{s_c}_{K_{\pm}}} \|P_{<1}u\|_{V^2_{K_{\pm}}}.
}
Also for $d \ge 5$, from \eqref{5dn5}, Remark \ref{estx}, \eqref{5di2a}, \eqref{5di2c} and \eqref{5di2d} to have   
\EQQS{
 |I_2| \lec T^{1/(d+1)} \|n\|_{\dot{Y}^{s'}_{W_{\pm c}}} \|v\|_{Y^{s'}_{K_{\pm}}}\|P_{<1}u\|_{V^2_{K_{\pm}}}.
}
We prove for $4 \le d < 5$ and $(u,v,n)$ are spherically symmetric functions.  
Due to the operator $P_{<1}$,  
\EQQS{  
 |I_2|  
   &\lec \Bigl|\int_{\R^{1+d}} \Bigl(\sum_{N_3 \lec 1}\til{n}_{N_3}\Bigr)\Bigl(\sum_{N_2 < 1}\om_1^{-1}\til{v}_{N_2}\Bigr)
                  (\ol{P_{<1}\til{u}})dxdt\Bigr|  \\
   &\qquad + \sum_{N_2 \ge 1} \Bigl|\int_{\R^{1+d}} \Bigl(\sum_{N_3 \lec N_2} \til{n}_{N_3}\Bigr)(\om_1^{-1}\til{v}_{N_2})
                     (\ol{P_{<1}\til{u}})dxdt\Bigr|  \\                                                                               
   &=: I_{2,1} + I_{2,2}. 
}
First, we estimate $I_{2,2}$. 
We apply the H\"{o}lder inequality to have  
\EQS{
 |I_{2,2}| \lec \sum_{N_2 \ge 1}\Bigl{\|}\sum_{N_3 \lec N_2}\til{n}_{N_3}\Bigr{\|}_{L^3_{t,x}} \|\om_1^{-1}\til{v}_{N_2}\|_{L^3_{t,x}} 
                   \|P_{<1}\til{u}\|_{L^3_{t,x}}.                                                                                      \label{i22a}
}
By Proposition \ref{rStr}, \eqref{c}, $d < 5, N_3 \lec N_2 \sim N_1$ 
and the Cauchy-Schwarz inequality, then we have  
\EQS{ 
 \Bigl{\|} \sum_{N_3 \lec N_2} \til{n}_{N_3}\Bigr{\|}_{L^3_{t,x}} 
   &\lec \Bigl{\|} |\na_x|^{(d-2)/6} \sum_{N_3 \lec N_2} \til{n}_{N_3}\Bigr{\|}_{V^2_{W_{\pm c}}} \notag \\
   &\lec \Bigl(\sum_{N_3 \lec N_2} N_3^{2(5-d)/3}\Bigr)^{1/2} 
              \Bigl(\sum_{N_3 \lec N_2} N_3^{d-4}\|n_{N_3}\|_{V^2_{W_{\pm c}}}^2\Bigr)^{1/2} \notag \\
   &\lec N_2^{(5-d)/3}\|n\|_{\dot{Y}^{s_c}_{W_{\pm c}}}.                                                                      \label{i22b} 
}
From $d \ge 4$, \eqref{4dr2} and \eqref{4dr3}, we see 
\EQS{
 \|P_{<1}\til{u}\|_{L^3_{t,x}} \lec \|P_{<1}u\|_{V^2_{K_{\pm}}}, \qquad 
 \| \om_1^{-1}\til{v}_{N_2}\|_{L^3_{t,x}} \lec \LR{N_2}^{(d-8)/6}\|v_{N_2}\|_{V^2_{K_{\pm}}}.                           \label{i22c} 
}
Collecting \eqref{i22a}--\eqref{i22c}, $N_2 \ge 1$ and applying the Cauchy-Schwarz inequality, we obtain 
\EQS{
 |I_{2,2}| &\lec \sum_{N_2 \ge 1} N_2^{(2-d)/6}\|n\|_{\dot{Y}^{s_c}_{W_{\pm c}}}\|v_{N_2}\|_{V^2_{K_{\pm}}}
                     \|P_{<1}u\|_{V^2_{K_{\pm}}} \notag \\
           &\lec \|n\|_{\dot{Y}^{s_c}_{W_{\pm c}}}\|v\|_{Y^{s_c}_{W_{\pm c}}}\|P_{<1}u\|_{V^2_{K_{\pm}}}.            \label{be1}
}
Next, we estimate $I_{2,1}$. 
By the H\"{o}lder inequality to have  
\EQS{
 |I_{2,1}| \lec \Bigl{\|}\sum_{N_3 \lec 1}\til{n}_{N_3}\Bigr{\|}_{L^3_{t,x}}
                   \Bigl{\|} \sum_{N_2 < 1}\om_1^{-1}\til{v}_{N_2} \Bigr{\|}_{L^3_{t,x}} 
                     \|P_{<1}\til{u}\|_{L^3_{t,x}}.                                                                                    \label{i21A}
}
From $4 \le d < 5$, \eqref{i22b}, \eqref{i22c} and discarding $\om_1^{-1}$, we see 
\EQS{
 \Bigl{\|} \sum_{N_3 \lec 1} \til{n}_{N_3}\Bigr{\|}_{L^3_{t,x}} \lec \|n\|_{\dot{Y}^{s_c}_{W_{\pm c}}}, \qquad 
 \Bigl{\|} \sum_{N_2 < 1}\om_1^{-1}\til{v}_{N_2}\Bigr{\|}_{L^3_{t,x}} \lec \|P_{<1}v\|_{V^2_{K_{\pm}}}.               \label{i21B}
} 
Collecting \eqref{i22c}--\eqref{i21B}, we have 
\EQS{
 |I_{2,1}| \lec \|n\|_{\dot{Y}^{s_c}_{W_{\pm c}}} \|P_{<1}v\|_{V^2_{K_{\pm}}}\|P_{<1}u\|_{V^2_{K_{\pm}}} 
          \lec \|n\|_{\dot{Y}^{s_c}_{W_{\pm c}}} \|v\|_{Y^{s_c}_{K_{\pm}}}\|P_{<1}u\|_{V^2_{K_{\pm}}}.               \label{be2}
}
From \eqref{be1} and \eqref{be2}, we obtain 
$|I_2| \lec \|n\|_{\dot{Y}^{s_c}_{W_{\pm c}}} \|v\|_{Y^{s_c}_{K_{\pm}}}\|P_{<1}u\|_{V^2_{K_{\pm}}}$. 
We prove $(iii)$ for $d=4$ below.  
We apply the H\"{o}lder inequality, the Sobolev inequality, Proposition \ref{Str} and \eqref{c}, then we have  
\EQS{
 |I_3| 
  &\lec \|\1_{[0,T)}\|_{L^4_t}\Bigl{\|}\sum_{N_3 \lec N_2}\til{n}_{N_3}\Bigr{\|}_{L^{20/9}_t L^5_x}
                   \|\om_1^{-1}\til{v}_{N_2}\|_{L^{20/3}_t L^{5/2}_x} \|\til{u}_{N_1}\|_{L^{20/3}_t L^{5/2}_x}  \notag \\
  &\lec T^{1/4}\Bigl{\|} |\na_x|^{3/4}\sum_{N_3 \lec N_2}n_{N_3}\Bigr{\|}_{V^2_{W_{\pm c}}} \LR{N_2}^{1/4-1}
                   \|v_{N_2}\|_{V^2_{K_{\pm}}} \LR{N_1}^{1/4}\|u_{N_1}\|_{V^2_{K_{\pm}}}.                      \label{2}
}
By the Cauchy-Schwarz inequality, we have 
\EQS{
 \Bigl{\|} |\na_x|^{3/4}\sum_{N_3 \lec N_2}n_{N_3}\Bigr{\|}_{V^2_{W_{\pm c}}}   
    &\lec \sum_{N_3 \lec N_2} N_3^{3/4}\|n_{N_3}\|_{V^2_{W_{\pm c}}} \notag \\
    &\lec \Bigl(\sum_{N_3 \lec N_2}N_3\Bigr)^{1/2}\Bigl(\sum_{N_3 \lec N_2}
               N_3^{1/2}\|n_{N_3}\|_{V^2_{W_{\pm c}}}^2\Bigr)^{1/2} \notag \\
    &\lec N_2^{1/2}\|n\|_{\dot{Y}^{1/4}_{W_{\pm c}}}.                                                                          \label{2'}
}
Collecting \eqref{2}, \eqref{2'} and $N_1 \sim N_2$, we obtain  
\EQQS{
  |I_3| \lec T^{1/4} \|n\|_{\dot{Y}^{1/4}_{W_{\pm c}}} \|v_{N_2}\|_{V^2_{K_{\pm}}} \|u_{N_1}\|_{V^2_{K_{\pm}}}.    
}
We prove for $d \ge 5$. 
We apply the H\"{o}lder inequality, \eqref{5dn3}, \eqref{5dn2}, Remark \ref{estx} and $N_1 \sim N_2$, then we have  
\EQS{
 |I_3|  
   &\lec \Bigl{\|}\sum_{N_3 \lec N_2} \til{n}_{N_3}\Bigr{\|}_{L^{(d+1)/2}_{t,x}} \|\om_1^{-1}\til{v}_{N_2}\|_{L^{2(d+1)/(d-1)}_{t,x}}
              \| \til{u}_{N_1}\|_{L^{2(d+1)/(d-1)}_{t,x}}                                                          \label{5dn7} \\
   &\lec \Bigl{\|} |\na_x|^{s_c}\sum_{N_3 \lec N_2} \til{n}_{N_3}\Bigr{\|}_{V^2_{W_{\pm c}}} 
              \LR{N_2}^{-1/2} \|v_{N_2}\|_{V^2_{K_{\pm}}} \LR{N_1}^{1/2} \|u_{N_1}\|_{V^2_{K_{\pm}}} \notag \\  
   &\lec \|n\|_{\dot{Y}^{s_c}_{W_{\pm c}}} \|v_{N_2}\|_{V^2_{K_{\pm}}} \|u_{N_1}\|_{V^2_{K_{\pm}}}.    \notag   
}
From \eqref{5dn5} and Remark \ref{estx},  we have  
\EQS{
 \Bigl{\|}\sum_{N_3 \lec N_2} \til{n}_{N_3}\Bigr{\|}_{L^{(d+1)/2}_{t,x}} 
   \lec T^{1/(d+1)}\|n\|_{\dot{Y}^{s'}_{W_{\pm c}}}.                                                          \label{5dn8}
}
Collecting \eqref{5dn2}, \eqref{5dn7} and \eqref{5dn8}, we obtain 
\EQQS{
 |I_3| \lec T^{1/(d+1)}\|n\|_{\dot{Y}^{s'}_{W_{\pm c}}} \|v_{N_2}\|_{V^2_{K_{\pm}}} \|u_{N_1}\|_{V^2_{K_{\pm}}}.
}
When $4 \le d < 5$ and $(u,v,n)$ are spherically symmetric functions, we apply the H\"{o}lder inequality to have  
\EQS{
 |I_3| \lec \Bigl{\|}\sum_{N_3 \lec N_2}\til{n}_{N_3}\Bigr{\|}_{L^3_{t,x}} 
                \|\om_1^{-1}\til{v}_{N_2}\|_{L^3_{t,x}} \|\til{u}_{N_1}\|_{L^3_{t,x}}. 
                                                                                                                    \label{4dr5}
}
From \eqref{4dr2}, \eqref{4dr3}, we see 
\EQS{
 \|\til{u}_{N_1}\|_{L^3_{t,x}} \lec \LR{N_2}^{(d-2)/6} \|u_{N_1}\|_{V^2_{K_{\pm}}}, \qquad 
 \| \om_1^{-1}\til{v}_{N_2}\|_{L^3_{t,x}} \lec \LR{N_2}^{(d-8)/6}\|v_{N_2}\|_{V^2_{K_{\pm}}}.      \label{4dr6} 
}
Collecting \eqref{4dr5}, \eqref{i22b} and \eqref{4dr6}, we obtain 
\EQQS{
 |I_3| \lec \|n\|_{\dot{Y}^{s_c}_{W_{\pm c}}} \|v_{N_2}\|_{V^2_{K_{\pm}}} \|u_{N_1}\|_{V^2_{K_{\pm}}}.    
}
We prove $(iv)$. 
The estimate for $I_6$ is the same manner as the estimate for $I_4$, so we only estimate $I_4, I_5$. 
First, we estimate $I_4$ for $d=4$. 
We apply the H\"{o}lder inequality, Proposition \ref{modulation}, the Sobolev inequality, 
Lemma \ref{p}, Proposition \ref{Str}, \eqref{c}, $\LR{N_1} \sim N_1 \ge 1$ and Lemma \ref{estX} to have   
\EQQS{
 |I_4|  
  &\lec \|Q_{\ge M}\til{n}_{N_3}\|_{L^2_{t,x}}\Bigl{\|}\sum_{N_2 \ll N_1}Q_2 \om_1^{-1}\til{v}_{N_2}\Bigr{\|}_{L^4_t L^{16/3}_x} 
             \|Q_1 \til{u}_{N_1}\|_{L^4_t L^{16/5}_x} \\
  &\lec N_1^{-1/2}\|\til{n}_{N_3}\|_{V^2_{W_{\pm c}}} \Bigl{\|} \LR{\na_x}^{5/4}\sum_{N_2 \ll N_1}Q_2 \om_1^{-1}
             \til{v}_{N_2}\Bigr{\|}_{L^4_t L^2_x} \| \LR{\na_x}^{1/12}Q_1 \til{u}_{N_1}\|_{L^4_t L^3_x} \\
  &\lec N_1^{-1/2}\|n_{N_3}\|_{V^2_{W_{\pm c}}} T^{1/4}\Bigl{\|} \LR{\na_x}^{5/4}\sum_{N_2 \ll N_1}\om_1^{-1}
             \til{v}_{N_2}\Bigr{\|}_{V^2_{K_{\pm}}} \LR{N_1}^{1/2}\|u_{N_1}\|_{V^2_{K_{\pm}}} \\
  &\lec T^{1/4}\|n_{N_3}\|_{V^2_{W_{\pm c}}} \|v\|_{Y^{1/4}_{K_{\pm}}} \|u_{N_1}\|_{V^2_{K_{\pm}}}.
}
Next, we prove for $d \ge 4$. 
We apply the H\"{o}lder inequality to have  
\EQS{
 |I_4|   
  \lec \|Q_{\ge M}\til{n}_{N_3}\|_{L^2_{t,x}}\Bigl{\|}\sum_{N_2 \ll N_1}Q_2 \om_1^{-1}\til{v}_{N_2}\Bigr{\|}_{L^{d+1}_{t,x}} 
             \|Q_1 \til{u}_{N_1}\|_{L^{2(d+1)/(d-1)}_{t,x}}.                                                               \label{I4}
}
By Proposition \ref{modulation}, Proposition \ref{Str}, \eqref{5dn2} and \eqref{c}, we have 
\EQS{
 &\|Q_{\ge M}\til{n}_{N_3}\|_{L^2_{t,x}} \lec N_1^{-1/2}\|n_{N_3}\|_{V^2_{W_{\pm c}}},                           \label{I4a} \\
 &\|Q_1 \til{u}_{N_1}\|_{L^{2(d+1)/(d-1)}_{t,x}} \lec \LR{N_1}^{1/2} \|u_{N_1}\|_{V^2_{K_{\pm}}}.              \label{I4b} 
}
We apply the Sobolev inequality, Proposition \ref{Str}, \eqref{c} and Lemma \ref{estX}, we have
\EQS{
 \Bigl{\|}\sum_{N_2 \ll N_1}Q_2 \om_1^{-1}\til{v}_{N_2}\Bigr{\|}_{L^{d+1}_{t,x}} 
  &\lec \Bigl{\|} \LR{\na_x}^{d(d-3)/2(d-1)}\sum_{N_2 \ll N_1}Q_2 \om_1^{-1}
               \til{v}_{N_2}\Bigr{\|}_{L^{d+1}_t L^{2(d^2-1)/(d^2-5)}_x}  \notag \\
  &\lec \Bigl{\|} \LR{\na_x}^{d(d-3)/2(d-1)+1/(d-1)-1}\sum_{N_2 \ll N_1}\til{v}_{N_2}\Bigr{\|}_{V^2_{K_{\pm}}} \notag \\
  &\lec \|v\|_{Y^{s_c}_{K_{\pm}}}.                                                                                        \label{I4c}
}
Collecting \eqref{I4}--\eqref{I4c} and $N_1 \ge 1$, we obtain  
\EQQS{
 |I_4| \lec \|n_{N_3}\|_{V^2_{W_{\pm c}}} \|v\|_{Y^{s_c}_{K_{\pm}}} \|u_{N_1}\|_{V^2_{K_{\pm}}}. 
} 
For $d \ge 5$, in \eqref{I4}, if we apply the Sobolev inequality, Lemma \ref{p} and Lemma \ref{estX} to have 
\EQS{
 \Bigl{\|}\sum_{N_2 \ll N_1}Q_2 \om_1^{-1}\til{v}_{N_2}\Bigr{\|}_{L^{d+1}_{t,x}} 
  &\lec \Bigl{\|} \LR{\na_x}^{d(d-1)/2(d+1)}\sum_{N_2 \ll N_1}Q_2 \om_1^{-1} \til{v}_{N_2}\Bigr{\|}_{L^{d+1}_t L^2_x}  \notag \\
  &\lec T^{1/(d+1)} \Bigl{\|} \LR{\na_x}^{(d^2-3d-2)/2(d+1)}\sum_{N_2 \ll N_1}\til{v}_{N_2}\Bigr{\|}_{V^2_{K_{\pm}}}  \notag \\
  &\lec T^{1/(d+1)} \|v\|_{Y^{s'}_{K_{\pm}}}.                                                                             \label{I4d}
}
Collecting \eqref{I4}--\eqref{I4b}, \eqref{I4d} and $N_1 \ge 1$, we obtain 
\EQQS{
 |I_4| \lec T^{1/(d+1)} \|n_{N_3}\|_{V^2_{W_{\pm c}}} \|v\|_{Y^{s'}_{K_{\pm}}} \|u_{N_1}\|_{V^2_{K_{\pm}}}.
}
Next, we prove $I_5$. 
When $d=4$, by the H\"{o}lder inequality, the Sobolev inequality, Lemma \ref{p}, Proposition \ref{Str}, 
\eqref{c}, Proposition \ref{modulation}, $N_1 \sim N_3$ and Lemma \ref{estX}, we have 
\EQQS{
 |I_5|  
  &\lec \|Q_3 \til{n}_{N_3}\|_{L^4_t L^{16/5}_x}\Bigl{\|}\sum_{N_2 \ll N_1}Q_{\ge M} 
             \om_1^{-1}\til{v}_{N_2}\Bigr{\|}_{L^2_t L^{16/3}_x} \|Q_1 \til{u}_{N_1}\|_{L^4_t L^2_x} \\
  &\lec \| |\na_x|^{1/12}Q_3 \til{n}_{N_3}\|_{L^4_t L^3_x}\Bigl{\|} \LR{\na_x}^{5/4}\sum_{N_2 \ll N_1}Q_{\ge M}
             \om_1^{-1}\til{v}_{N_2}\Bigr{\|}_{L^2_{t,x}} T^{1/4}\|u_{N_1}\|_{V^2_{K_{\pm}}} \\
  &\lec N_3^{1/2}\|n_{N_3}\|_{V^2_{W_{\pm c}}} N_1^{-1/2} \Bigl{\|} \LR{\na_x}^{5/4}\sum_{N_2 \ll N_1}\om_1^{-1}
              \til{v}_{N_2}\Bigr{\|}_{V^2_{K_{\pm}}} T^{1/4}\|u_{N_1}\|_{V^2_{K_{\pm}}} \\
  &\lec T^{1/4}\|n_{N_3}\|_{V^2_{W_{\pm c}}} \|v\|_{Y^{1/4}_{K_{\pm}}} \|u_{N_1}\|_{V^2_{K_{\pm}}}.
}
For $d \ge 5$, by the H\"{o}lder inequality, we have 
\EQS{
 |I_5|  
  \lec \|Q_3 \til{n}_{N_3}\|_{L^{2(d+1)/(d-1)}_{t,x}}\Bigl{\|}\sum_{N_2 \ll N_1}Q_{\ge M} 
             \om_1^{-1}\til{v}_{N_2}\Bigr{\|}_{L^2_t L^{d+1}_x} 
                \|Q_1 \til{u}_{N_1}\|_{L^{d+1}_t L^2_x}.                                                                    \label{I5}
}
From \eqref{I4b}, $N_3 \sim N_1$ and Lemma \ref{p}, we have 
\EQS{
 &\|Q_3 \til{n}_{N_3}\|_{L^{2(d+1)/(d-1)}_{t,x}} \lec \LR{N_1}^{1/2} \|n_{N_3}\|_{V^2_{W_{\pm c}}},            \label{I5a} \\
 &\|Q_1 \til{u}_{N_1}\|_{L^{d+1}_t L^2_x} \lec T^{1/(d+1)} \|u_{N_1}\|_{V^2_{K_{\pm}}}.                          \label{I5b} 
}
We apply the Sobolev inequality, Proposition \ref{modulation}, \eqref{c} and Lemma \ref{estX}, we have 
\EQS{
 \Bigl{\|}\sum_{N_2 \ll N_1}Q_{\ge M} \om_1^{-1}\til{v}_{N_2}\Bigr{\|}_{L^2_t L^{d+1}_x} 
  &\lec \Bigl{\|} \LR{\na_x}^{d(d-1)/2(d+1)}\sum_{N_2 \ll N_1} Q_{\ge M}\om_1^{-1}\til{v}_{N_2}\Bigr{\|}_{L^2_{t,x}} \notag \\
  &\lec N_1^{-1/2} \Bigl{\|} \LR{\na_x}^{(d^2-3d-2)/2(d+1)} \sum_{N_2 \ll N_1}v_{N_2}\Bigr{\|}_{V^2_{K_{\pm}}} \notag \\
  &\lec N_1^{-1/2} \|v\|_{Y^{s'}_{K_{\pm}}}.                                                                              \label{I5c}
}
Collecting \eqref{I5}--\eqref{I5c} and $N_1 \ge 1$, we obtain 
\EQQS{
 |I_5| \lec T^{1/(d+1)}\|n_{N_3}\|_{V^2_{W_{\pm c}}} \|v\|_{Y^{s'}_{K_{\pm}}} \|u_{N_1}\|_{V^2_{K_{\pm}}}. 
}
Finally, we prove for $(u,v,n)$ are spherically symmetric functions and $d \ge 4$.  
By the H\"{o}lder inequality, Proposition \ref{rStr}, \eqref{c}, the Sobolev inequality, 
$1 \le N_1 \sim N_3$, Proposition \ref{modulation} and Lemma \ref{estX}, we have 
\EQQS{
 |I_5| 
  &\lec \|Q_3 \til{n}_{N_3}\|_{L^4_t L^{2d/(d-1)}_x}\Bigl{\|}\sum_{N_2 \ll N_1}Q_{\ge M} 
                     \om_1^{-1}\til{v}_{N_2}\Bigr{\|}_{L^2_t L^d_x} \|Q_1 \til{u}_{N_1}\|_{L^4_t L^{2d/(d-1)}_x} \\
  &\lec N_3^{1/4}\|n_{N_3}\|_{V^2_{W_{\pm c}}}\Bigl{\|} \LR{\na_x}^{(d-2)/2}\sum_{N_2 \ll N_1}
             Q_{\ge M} \om_1^{-1}\til{v}_{N_2}\Bigr{\|}_{L^2_{t,x}} \LR{N_1}^{1/4}\|u_{N_1}\|_{V^2_{K_{\pm}}} \\
  &\lec N_1^{1/2}\|n_{N_3}\|_{V^2_{W_{\pm c}}} N_1^{-1/2} \Bigl{\|} \LR{\na_x}^{(d-2)/2}\sum_{N_2 \ll N_1} \om_1^{-1} 
             \til{v}_{N_2}\Bigr{\|}_{V^2_{K_{\pm}}} \|u_{N_1}\|_{V^2_{K_{\pm}}} \\
  &\lec \|n_{N_3}\|_{V^2_{W_{\pm c}}} \|v\|_{Y^{s_c}_{K_{\pm}}} \|u_{N_1}\|_{V^2_{K_{\pm}}}.
}
\end{proof}

\section{Bilinear estimates}
\begin{prop}  \label{BE}
Let $(\th, s) = (1/4,1/4)$ for $d = 4$ and 
$(\th, s) = (1/(d+1), (d^2-3d-2)/2(d+1))$ for $d \ge 5$. For sufficiently small $T > 0$, 
\EQS{
 &\| I_{T,K_{\pm}}(n,v)\|_{Z^s_{K_{\pm}}} 
      \lec T^{\th} \|n\|_{\dot{Y}^s_{W_{\pm c}}} \|v\|_{Y^s_{K_{\pm}}},                     \label{BEKG} \\
 &\| I_{T,\, W_{\pm c}}(u,v)\|_{\dot{Z}^s_{W_{\pm c}}} 
      \lec T^{\th} \|u\|_{Y^s_{K_{\pm}}} \|v\|_{Y^s_{K_{\pm}}}.                               \label{BEW}
}
Moreover, we assume $(u,v,n)$ are spherically symmetric functions. 
Then for $d \ge 4$ and for all $0 < T < \I$, \eqref{BEKG}, \eqref{BEW} also holds with $(\th, s) = (0, d/2-2)$.   
\end{prop}
\begin{proof}
We denote $\til{u}_{N_1}:=\1_{[0,T)}P_{N_1}u, \til{v}_{N_2}:=\1_{[0,T)}P_{N_2}v, \til{n}_{N_3}:=\1_{[0,T)}P_{N_3}n$. 
First, we prove \eqref{BEKG}. 
\EQQS{
  \| I_{T,K_{\pm}}(n,v)\|_{Z^s_{K_{\pm}}}^2 \lec \sum_{i=0}^3 J_i   
}
where 
\EQQS{    
 &J_0 := \Bigl{\|} \int_0^t \1_{[0,T)}(t')K_{\pm}(t-t') P_{<1}(\til{n} (\om_1^{-1} \til{v}))(t')dt'\Bigr{\|}_{U^2_{K_{\pm}}}^2, \\
 &J_1 := \sum_{N_1 \ge 1} N_1^{2s} \Bigl{\|} \int_0^t \1_{[0,T)}(t') K_{\pm}(t-t') \sum_{N_2 \sim N_1} \sum_{N_3 \lec N_2} 
              P_{N_1}(\til{n}_{N_3}(\om_1^{-1}\til{v}_{N_2}))(t')dt' \Bigr{\|}_{U^2_{W_\pm}}^2,  \\
 &J_2 := \sum_{N_1 \ge 1} N_1^{2s} \Bigl{\|} \int_0^t \1_{[0,T)}(t') K_{\pm}(t-t') \sum_{N_2 \ll N_1} \sum_{N_3 \sim N_1} 
              P_{N_1}(\til{n}_{N_3}(\om_1^{-1}\til{v}_{N_2}))(t')dt' \Bigr{\|}_{U^2_{W_\pm}}^2,  \\  
 &J_3 := \sum_{N_1 \ge 1} N_1^{2s} \Bigl{\|} \int_0^t \1_{[0,T)}(t') K_{\pm}(t-t') \sum_{N_2 \gg N_1} \sum_{N_3 \sim N_2} 
              P_{N_1}(\til{n}_{N_3}(\om_1^{-1}\til{v}_{N_2}))(t')dt' \Bigr{\|}_{U^2_{W_\pm}}^2. 
}
By Corollary \ref{U2A}, we have 
\EQS{
 J_0^{1/2} 
      \lec \sup_{\|u\|_{V^2_{K_{\pm}}}=1} \Bigl{|}\int_{\R^{1+d}} \til{n}(\om_1^{-1}\til{v})(\ol{P_{<1}\til{u}})dxdt\Bigr{|}.   \label{j0}
}
For $d=4$ and $s=1/4$, from \eqref{j0}, Lemma \ref{tri} $(ii)$ and $\|P_{<1}u\|_{V^2_{K_{\pm}}} \lec \|u\|_{V^2_{K_{\pm}}}$, 
we obtain  
\EQQS{
 J_0^{1/2} \lec T\|n\|_{\dot{Y}^{1/4}_{W_{\pm c}}}\|v\|_{Y^{1/4}_{K_{\pm}}}.
}
We apply Corollary \ref{U2A} to have  
\EQS{
 J_1 \lec \sum_{N_1 \ge 1} N_1^{2s} \sup_{\|u\|_{V^2_{K_{\pm}}}=1} \Bigl{|}\sum_{N_2 \sim N_1} \sum_{N_3 \lec N_2}
                 \int_{\R^{1+d}} \til{n}_{N_3} (\om_1^{-1}\til{v}_{N_2}) \ol{\til{u}_{N_1}} dxdt\Bigr{|}^2.                           \label{j1}
}
For $d=4$ and $s=1/4$, by \eqref{j1}, $N_1 \sim N_2$, Lemma \ref{tri} $(iii)$ and 
$\|u_{N_1}\|_{V^2_{K_{\pm}}} \lec \|u\|_{V^2_{K_{\pm}}}$, we have  
\EQQS{
 J_1 \lec \sum_{N_2 \gec 1}N_2^{1/2}T^{1/2}\|n\|_{\dot{Y}^{1/4}_{W_{\pm c}}}^2 \|v_{N_2}\|_{V^2_{K_{\pm}}}^2 
      \lec T^{1/2}\|n\|_{\dot{Y}^{1/4}_{W_{\pm c}}}^2 \|v\|_{Y^{1/4}_{K_{\pm}}}^2.
}
For the estimate of $J_2$, we take $M=\e N_1$ for sufficiently small $\e >0$. 
Then, from Lemma \ref{Recovery}, we have  
\EQQS{
   &Q_{<M} \bigl( (Q_{<M} \til{n}_{N_3})(Q_{<M} \om_1^{-1}\til{v}_{N_2}) \bigr) \\
   &= Q_{<M} \Bigl[ \F^{-1}\Bigl( \int_{\ta_1 = \ta_2 + \ta_3,\, \xi_1 = \xi_2 + \xi_3} 
       \widehat{(Q_{<M} \til{n}_{N_3})}(\ta_3, \xi_3) \widehat{(Q_{<M} \om_1^{-1}\til{v}_{N_2})}(\ta_2, \xi_2) \Bigr) \Bigr]
   =0
}
when $N_1 \gg N_2$. 
Therefore, 
\EQQS{
 \til{n}_{N_3} (\om_1^{-1}\til{v}_{N_2}) = \sum_{i=1}^3 F_i, 
}
where 
\EQQS{
 &F_1 := Q_1\bigl( (Q_{\ge M} \til{n}_{N_3}) (Q_2 \om_1^{-1}\til{v}_{N_2})\bigr), \qquad  
   F_2 := Q_1\bigl( (Q_3 \til{n}_{N_3}) (Q_{\ge M} \om_1^{-1}\til{v}_{N_2})\bigr), \notag \\
 &F_3 := Q_{\ge M}\bigl((Q_3 \til{n}_{N_3}) (Q_2 \om_1^{-1}\til{v}_{N_2})\bigr).
}
For the estimate of $F_1$, 
we apply Corollary \ref{U2A} to have 
\EQS{
 &\sum_{N_1 \ge 1} N_1^{2s} \Bigl{\|} \int_0^t \1_{[0,T)}(t') K_{\pm}(t-t') \sum_{N_2 \ll N_1} \sum_{N_3 \sim N_1} 
            P_{N_1}F_1(t') dt' \Bigr{\|}_{U^2_{W_\pm}}^2  \notag \\
 &\lec \sum_{N_1 \ge 1} N_1^{2s} \sup_{\|u\|_{V^2_{K_{\pm}}}=1} \Bigl{|} \sum_{N_2 \ll N_1} \sum_{N_3 \sim N_1} 
            \int_{\R^{1+d}} (Q_{\ge M}\til{n}_{N_3}) (Q_2 \om_1^{-1}\til{v}_{N_2} ) (\ol{Q_1 \til{u}_{N_1}}) dxdt\Bigr{|}^2.  \label{j2a}
}
For $d=4, s=1/4$, from Lemma \ref{tri} $(iv), N_3 \sim N_1 \ge 1$ and 
$\|u_{N_1}\|_{V^2_{K_{\pm}}} \lec \|u\|_{V^2_{K_{\pm}}}$, the right-hand side of \eqref{j2a} is bounded by  
\EQS{
   T^{1/2} \sum_{N_3 \gec 1} N_3^{1/2} \|n_{N_3}\|_{V^2_{W_{\pm c}}}^2 \|v\|_{Y^{1/4}_{K_{\pm}}}^2 
   \lec T^{1/2} \|n\|_{\dot{Y}^{1/4}_{W_{\pm c}}}^2 \|v\|_{Y^{1/4}_{K_{\pm}}}^2.                                         \label{J21}
}
For the estimate of $F_2$, 
we apply Corollary \ref{U2A} to have 
\EQS{
 &\sum_{N_1 \ge 1} N_1^{2s} \Bigl{\|} \int_0^t \1_{[0,T)}(t') K_{\pm}(t-t') \sum_{N_2 \ll N_1} \sum_{N_3 \sim N_1} 
    P_{N_1}F_2(t') dt' \Bigr{\|}_{U^2_{W_\pm}}^2 \notag \\
 &\lec \sum_{N_1 \ge 1} N_1^{2s} \sup_{\|u\|_{V^2_{K_{\pm}}}=1} \Bigl{|} \sum_{N_2 \ll N_1} \sum_{N_3 \sim N_1} 
            \int_{\R^{1+d}} (Q_3 \til{n}_{N_3}) (Q_{\ge M} \om_1^{-1}\til{v}_{N_2} ) (\ol{Q_1 \til{u}_{N_1}}) dxdt\Bigr{|}^2.  \label{j2b}
}
For $d=4, s=1/4$, we apply Lemma \ref{tri} $(iv), N_3 \sim N_1 \ge 1$ and 
$\|u_{N_1}\|_{V^2_{K_{\pm}}} \lec \|u\|_{V^2_{K_{\pm}}}$, 
then the right-hand side of \eqref{j2b} is bounded by     
\EQS{
 T^{1/2} \sum_{N_3 \gec 1} N_3^{1/2} \|n_{N_3}\|_{V^2_{W_{\pm c}}}^2 \|v\|_{Y^{1/4}_{K_{\pm}}}^2 
   \lec T^{1/2} \|n\|_{\dot{Y}^{1/4}_{W_{\pm c}}}^2 \|v\|_{Y^{1/4}_{K_{\pm}}}^2.                                         \label{J22}
}
For the estimate for $F_3$, 
we apply Corollary \ref{U2A} to have 
\EQS{
 &\sum_{N_1 \ge 1} N_1^{2s} \Bigl{\|} \int_0^t \1_{[0,T)}(t') K_{\pm}(t-t') \sum_{N_2 \ll N_1} \sum_{N_3 \sim N_1} 
    P_{N_1}F_3(t') dt' \Bigr{\|}_{U^2_{W_\pm}}^2 \notag \\
 &\lec \sum_{N_1 \ge 1} N_1^{2s} \sup_{\|u\|_{V^2_{K_{\pm}}}=1} \Bigl{|} \sum_{N_2 \ll N_1} \sum_{N_3 \sim N_1} 
            \int_{\R^{1+d}} (Q_3 \til{n}_{N_3}) (Q_2 \om_1^{-1}\til{v}_{N_2} ) (\ol{Q_{\ge M} \til{u}_{N_1}}) dxdt\Bigr{|}^2.   \label{j2c}
}
For $d=4, s=1/4$, we apply Lemma \ref{tri} $(iv), N_3 \sim N_1 \ge 1$ and 
$\|u_{N_1}\|_{V^2_{K_{\pm}}} \lec \|u\|_{V^2_{K_{\pm}}}$, 
then the right-hand side of \eqref{j2c} is bounded by     
\EQS{
 T^{1/2} \sum_{N_3 \gec 1} N_3^{1/2} \|n_{N_3}\|_{V^2_{W_{\pm c}}}^2 \|v\|_{Y^{1/4}_{K_{\pm}}}^2 
   \lec T^{1/2} \|n\|_{\dot{Y}^{1/4}_{W_{\pm c}}}^2 \|v\|_{Y^{1/4}_{K_{\pm}}}^2.                                         \label{J23}
}
Collecting \eqref{J21}, \eqref{J22} and \eqref{J23}, we obtain 
$J_2 \lec T^{1/2}\|n\|_{\dot{Y}^{1/4}_{W_{\pm c}}}^2 \|v\|_{Y^{1/4}_{K_{\pm}}}^2$. 
By Corollary \ref{U2A} and the triangle inequality to have 
\EQS{
 J_3 &\lec \sum_{N_1 \ge 1} N_1^{2s} \sup_{\|u\|_{V^2_{K_{\pm}}}=1} \Bigl{|}\sum_{N_2 \gg N_1} \sum_{N_3 \sim N_2} 
               \int_{\R^{1+d}} \til{n}_{N_3} (\om_1^{-1}\til{v}_{N_2})\ol{\til{u}_{N_1}} dxdt\Bigr{|}^2 \notag \\ 
      &\lec \sum_{N_1 \ge 1} N_1^{2s} \Bigl{(}\sum_{N_2 \gg N_1} \sum_{N_3 \sim N_2} \sup_{\|u\|_{V^2_{K_{\pm}}}=1} \Bigl{|}
               \int_{\R^{1+d}} \til{n}_{N_3} (\om_1^{-1}\til{v}_{N_2})\ol{\til{u}_{N_1}} dxdt\Bigr{|}\Bigr{)}^2.                   \label{j3a}
}
By the same manner as the estimate for Lemma \ref{tri} $(iii)$, 
for $d=4, s=1/4$, we have 
\EQS{
 \Bigl{|}\int_{\R^{1+4}} \til{n}_{N_3} (\om_1^{-1}\til{v}_{N_2})\ol{\til{u}_{N_1}} dxdt\Bigr{|} 
  \lec T^{1/4} N_3^{1/4} \|n_{N_3}\|_{V^2_{W_{\pm c}}} \|v_{N_2}\|_{V^2_{K_{\pm}}} \|u_{N_1}\|_{V^2_{K_{\pm}}}.   \label{j3b}
}
From \eqref{j3b}, the right-hand side of \eqref{j3a} is bounded by  
\EQQS{
 \sum_{N_1 \ge 1} \Bigl{(}\sum_{N_2 \gg N_1} \sum_{N_3 \sim N_2} N_1^{1/4} T^{1/4} N_3^{1/4} \|n_{N_3}\|_{V^2_{W_{\pm c}}} 
       \|v_{N_2}\|_{V^2_{K_{\pm}}}\Bigr{)}^2.
}
Hence, $\| \cdot \|_{l^2 l^1} \lec \| \cdot \|_{l^1 l^2}$ and the Cauchy-Schwarz inequality to have 
\EQQS{
 J_3^{1/2} &\lec \sum_{N_2 \gec 1} \sum_{N_3 \sim N_2} \Bigl{(}\sum_{N_1 \ll N_2} N_1^{1/2} T^{1/2} N_3^{1/2} 
                        \|n_{N_3}\|_{V^2_{W_{\pm c}}}^2  \|v_{N_2}\|_{V^2_{K_{\pm}}}^2\Bigr{)}^{1/2} \\
              &\lec T^{1/4} \sum_{N_2 \gec 1} \sum_{N_3 \sim N_2} N_2^{1/4}N_3^{1/4}
                        \|n_{N_3}\|_{V^2_{W_{\pm c}}} \|v_{N_2}\|_{V^2_{K_{\pm}}} \\ 
              &\lec T^{1/4} \|n\|_{\dot{Y}^{1/4}_{W_{\pm c}}} \|v\|_{Y^{1/4}_{K_{\pm}}}. 
}

We prove \eqref{BEW}.
By Corollary \ref{U2A}, we only need to estimate $K_i \ (i=1,2,3)$:  
\EQQS{
 &K_1 := \sum_{N_3} N_3^{2s} \sup_{\|n\|_{V^2_{W_{\pm c}}}=1} \Bigl{|}\sum_{N_2 \sim N_3} \sum_{N_1 \ll N_3} 
              \int_{\R^{1+d}} (\om_1^{-1}\til{u}_{N_1})(\ol{\om_1^{-1}\til{v}_{N_2}})(\ol{\om \til{n}_{N_3}})dxdt\Bigr{|}^2,  \\
 &K_2 := \sum_{N_3} N_3^{2s} \sup_{\|n\|_{V^2_{W_{\pm c}}}=1} \Bigl{|}\sum_{N_2 \ll N_3} \sum_{N_1 \sim N_3} 
              \int_{\R^{1+d}} (\om_1^{-1}\til{u}_{N_1})(\ol{\om_1^{-1}\til{v}_{N_2}})(\ol{\om \til{n}_{N_3}})dxdt\Bigr{|}^2,  \\
 &K_3 := \sum_{N_3} N_3^{2s} \sup_{\|n\|_{V^2_{W_{\pm c}}}=1} \Bigl{|}\sum_{N_2 \gec N_3} \sum_{N_1 \sim N_2} 
              \int_{\R^{1+d}} (\om_1^{-1}\til{u}_{N_1})(\ol{\om_1^{-1}\til{v}_{N_2}})(\ol{\om \til{n}_{N_3}})dxdt\Bigr{|}^2.
}
First, we estimate $K_1$. 
Put $K_1 = K_{1,1} + K_{1,2}$ where 
\EQS{
 &K_{1,1} := \sum_{N_3 \lec 1} N_3^{2s} \sup_{\|n\|_{V^2_{W_{\pm c}}}=1} \Bigl{|}\sum_{N_2 \sim N_3} \sum_{N_1 \ll N_3} 
                \int_{\R^{1+d}} (\om_1^{-1}\til{u}_{N_1})(\ol{\om_1^{-1}\til{v}_{N_2}}) \notag \\
 &\qquad \qquad \qquad \qquad \qquad \qquad \qquad \cross (\ol{\om \til{n}_{N_3}}) dxdt\Bigr{|}^2, 
                                                                                                              \label{K11} \\
 &K_{1,2} := \sum_{N_3 \gg 1} N_3^{2s} \sup_{\|n\|_{V^2_{W_{\pm c}}}=1} \Bigl{|}\sum_{N_2 \sim N_3} \sum_{N_1 \ll N_3} 
                \int_{\R^{1+d}} (\om_1^{-1}\til{u}_{N_1})(\ol{\om_1^{-1}\til{v}_{N_2}}) (\ol{\om \til{n}_{N_3}}) dxdt\Bigr{|}^2. 
                                                                                                                                           \notag   
}
For $d=4, s=1/4$, by the same manner as the estimate for Lemma \ref{tri} $(i)$ and $N_1 \ll N_3 \lec 1$, we find 
\EQQS{
  &\Bigl{|}\int_{\R^{1+4}} \Bigl{(}\sum_{N_1 \ll N_3}\om_1^{-1}\til{u}_{N_1}\Bigr{)}
         (\ol{\om_1^{-1}\til{v}_{N_2}})(\ol{\om \til{n}_{N_3}})dxdt\Bigr{|} \\
  &\lec \|\1_{[0,T)}\|_{L^4_t} \Bigl{\|}\sum_{N_1 \ll N_3}\om_1^{-1}\til{u}_{N_1}\Bigr{\|}_{L^{20/3}_t L^{5/2}_x} 
             \| \om_1^{-1}\til{v}_{N_2}\|_{L^{20/3}_t L^{5/2}_x} \| \om \til{n}_{N_3}\|_{L^{20/9}_t L^5_x} \\
  &\lec T^{1/4} \Bigl{\|} \LR{\na_x}^{1/4}\sum_{N_1 \ll N_3}\om_1^{-1}\til{u}_{N_1}\Bigr{\|}_{V^2_{K_{\pm}}} 
             \| \LR{\na_x}^{1/4} \om_1^{-1}\til{v}_{N_2}\|_{V^2_{K_{\pm}}} \| |\na_x|^{3/4}\om \til{n}_{N_3}\|_{V^2_{W_{\pm c}}} \\
  &\lec T^{1/4} \LR{N_1}^{-1} \Bigl{\|} \LR{\na_x}^{1/4}\sum_{N_1 \ll N_3} \til{u}_{N_1}\Bigr{\|}_{V^2_{K_{\pm}}} 
             \LR{N_2}^{1/4-1} \|v_{N_2}\|_{V^2_{K_{\pm}}} N_3^{3/4+1} \|n_{N_3}\|_{V^2_{W_{\pm c}}} \\
  &\lec T^{1/4} \|u\|_{Y^{1/4}_{K_{\pm}}} \|v_{N_2}\|_{V^2_{K_{\pm}}} \|n_{N_3}\|_{V^2_{W_{\pm c}}}.       
}
Hence, 
\EQQS{
 K_{1,1} \lec \sum_{N_2 \lec 1} N_2^{1/2} (T^{1/4}\|u\|_{Y^{1/4}_{K_{\pm}}} \|v_{N_2}\|_{V^2_{K_{\pm}}})^2 
          \lec T^{1/2} \|u\|_{Y^{1/4}_{K_{\pm}}}^2 \|v\|_{Y^{1/4}_{K_{\pm}}}^2.  
} 
We take $M=\e N_2$ for sufficiently small $\e >0$. 
Then, from Lemma \ref{Recovery}, we have
\EQQS{
 &Q_{<M} \om_1^{-1}\1_{[0,T)}\bigl((Q_{<M} \om_1^{-1}\til{v}_{N_2})(Q_{<M}{\om \, \til{n}_{N_3}})\bigr) \\
 &= Q_{<M} \om_1^{-1}\1_{[0,T)}\Bigl[ \F^{-1}\Bigl( \int_{\ta_1 = \ta_2 + \ta_3,\, \xi_1 = \xi_2 + \xi_3} 
       \widehat{(Q_{<M} \om_1^{-1}\til{v}_{N_2})}(\ta_2, \xi_2)
       \widehat{(Q_{<M} \om \, \til{n}_{N_3})}(\ta_3, \xi_3)  \Bigr) \Bigr] \\
 &=0.
}
Therefore, 
\EQQS{
 (\om_1^{-1}\til{v}_{N_2}) (\om \, \til{n}_{N_3}) = \sum_{i=1}^3 G_i, 
}
where 
\EQQS{
 &G_1 :=  Q_{\ge M} \bigl((Q_2 \om_1^{-1}\til{v}_{N_2})(Q_3 \om \, \til{n}_{N_3})\bigr), \qquad 
   G_2 :=  Q_1 \bigl( (Q_{\ge M} \om_1^{-1}\til{v}_{N_2})(Q_3 \om \, \til{n}_{N_3}) \bigr), \notag \\
 &G_3 :=  Q_1 \bigl( (Q_2 \om_1^{-1}\til{v}_{N_2})(Q_{\ge M} \om \, \til{n}_{N_3})\bigr).
}
Hence, it follows that 
\EQQS{
 K_{1,2} \le K_{1,2,1} + K_{1,2,2} + K_{1,2,3}
}
where  
\EQS{
 K_{1,2,1} &:= \sum_{N_3 \gg 1} N_3^{2s} \sup_{\|n\|_{V^2_{W_{\pm c}}}=1} \Bigl{|}\sum_{N_2 \sim N_3}       
                        \sum_{N_1 \ll N_3} \int_{\R^{1+d}} (\om_1^{-1}\til{u}_{N_1}) \ol{G_1} dxdt\Bigr{|}^2,        \label{K121a} \\ 
 K_{1,2,2} &:= \sum_{N_3 \gg 1} N_3^{2s} \sup_{\|n\|_{V^2_{W_{\pm c}}}=1} \Bigl{|}\sum_{N_2 \sim N_3}
                        \sum_{N_1 \ll N_3} \int_{\R^{1+d}} (\om_1^{-1}\til{u}_{N_1}) \ol{G_2} dxdt\Bigr{|}^2,        \label{K122a} \\
 K_{1,2,3} &:= \sum_{N_3 \gg 1} N_3^{2s} \sup_{\|n\|_{V^2_{W_{\pm c}}}=1} \Bigl{|}\sum_{N_2 \sim N_3}
                        \sum_{N_1 \ll N_3} \int_{\R^{1+d}} (\om_1^{-1}\til{u}_{N_1}) \ol{G_3} dxdt\Bigr{|}^2.        \label{K123a}  
}  
By Lemma \ref{Mhigh}, 
\EQS{
 K_{1,2,1} &\lec  \sum_{N_3 \gg 1} N_3^{2s} \sup_{\|n\|_{V^2_{W_{\pm c}}}=1} \Bigl{|}\sum_{N_2 \sim N_3}
                        \sum_{N_1 \ll N_3} \int_{\R^{1+d}} (Q_{\ge M}\om_1^{-1}\til{u}_{N_1})(\ol{Q_2 \om_1^{-1}\til{v}_{N_2}}) 
                           \notag \\
            &\qquad \qquad       \cross (\ol{Q_3 \om \til{n}_{N_3}}) dxdt\Bigr{|}^2.                                          \label{k121}  
}
By the same manner as the estimate for Lemma \ref{tri} $(iv), i=5$, for $d=4, s=1/4$, we find  
\EQS{
 &\Bigl{|} \int_{\R^{1+4}} \Bigl{(}\sum_{N_1 \ll N_3}Q_{\ge M} \om_1^{-1} \til{u}_{N_1}\Bigr{)}(\ol{Q_2  \om_1^{-1} \til{v}_{N_2}})
     (\ol{Q_3 \om \til{n}_{N_3}})dxdt\Bigr{|} \notag \\ 
 &\lec T^{1/4}\|u\|_{Y^{1/4}_{K_{\pm}}}\|v_{N_2}\|_{V^2_{K_{\pm}}}\|n_{N_3}\|_{V^2_{W_{\pm c}}}.                   \label{k121a}
}
$N_3 \gg 1, N_2 \sim N_3$ implies $N_2 \ge 1$. 
Hence, from \eqref{k121} and \eqref{k121a}, we have 
\EQQS{
 K_{1,2,1} 
  &\lec  \sum_{N_3 \gg 1} N_3^{1/2} \sup_{\|n\|_{V^2_{K_{\pm}}}=1} \Bigl{|}\sum_{N_2 \sim N_3} 
                \sum_{N_1 \ll N_3} \int_{\R^{1+4}} (Q_{\ge M} \om_1^{-1} \til{u}_{N_1})(\ol{Q_2  \om_1^{-1} \til{v}_{N_2}}) \\
  &\qquad \qquad    \cross (\ol{Q_3 \om \til{n}_{N_3}})dxdt\Bigr{|}^2 \\
  &\lec \sum_{N_2 \ge 1} N_2^{1/2}(T^{1/4}\|u\|_{Y^{1/4}_{K_{\pm}}} \|v_{N_2}\|_{V^2_{K_{\pm}}})^2  
    \lec T^{1/2} \|u\|_{Y^{1/4}_{K_{\pm}}}^2 \|v\|_{Y^{1/4}_{K_{\pm}}}^2.
}
By Lemma \ref{Mhigh}, 
\EQS{
 K_{1,2,2} &\lec  \sum_{N_3 \gg 1} N_3^{2s} \sup_{\|n\|_{V^2_{W_{\pm c}}}=1} \Bigl{|}\sum_{N_2 \sim N_3}
                        \sum_{N_1 \ll N_3} \int_{\R^{1+d}} (Q_1 \om_1^{-1}\til{u}_{N_1})(\ol{Q_{\ge M} \om_1^{-1}\til{v}_{N_2}}) 
                           \notag \\
            &\qquad \qquad       \cross (\ol{Q_3 \om \til{n}_{N_3}}) dxdt\Bigr{|}^2.                                          \label{k122}  
}
By the same manner as the estimate for Lemma \ref{tri} $(iv), i=6$, for $d=4, s=1/4$, we find 
\EQS{
 &\Bigl{|} \int_{\R^{1+4}} \Bigl{(}\sum_{N_1 \ll N_3}Q_1\om_1^{-1} \til{u}_{N_1}\Bigr{)}(\ol{Q_{\ge M}  \om_1^{-1} \til{v}_{N_2}})
     (\ol{Q_3 \om \til{n}_{N_3}})dxdt\Bigr{|} \notag \\ 
 &\lec T^{1/4}\|u\|_{Y^{1/4}_{K_{\pm}}}\|v_{N_2}\|_{V^2_{K_{\pm}}}\|n_{N_3}\|_{V^2_{W_{\pm c}}}.                   \label{k122a}
}
$N_3 \gg 1, N_2 \sim N_3$ implies $N_2 \ge 1$. 
Hence, from \eqref{k122} and \eqref{k122a}, we have 
\EQQS{
 K_{1,2,2} 
  &\lec  \sum_{N_3 \gg 1} N_3^{1/2} \sup_{\|n\|_{V^2_{K_{\pm}}}=1} \Bigl{|}\sum_{N_2 \sim N_3} 
                \sum_{N_1 \ll N_3} \int_{\R^{1+4}} (Q_1 \om_1^{-1} \til{u}_{N_1})(\ol{Q_{\ge M} \om_1^{-1} \til{v}_{N_2}}) \\
  &\qquad \qquad    \cross (\ol{Q_3 \om \til{n}_{N_3}})dxdt\Bigr{|}^2 \\
  &\lec \sum_{N_2 \ge 1} N_2^{1/2}(T^{1/4}\|u\|_{Y^{1/4}_{K_{\pm}}} \|v_{N_2}\|_{V^2_{K_{\pm}}})^2  
    \lec T^{1/2} \|u\|_{Y^{1/4}_{K_{\pm}}}^2 \|v\|_{Y^{1/4}_{K_{\pm}}}^2.
}
By Lemma \ref{Mhigh}, 
\EQS{
 K_{1,2,3} &\lec  \sum_{N_3 \gg 1} N_3^{2s} \sup_{\|n\|_{V^2_{W_{\pm c}}}=1} \Bigl{|}\sum_{N_2 \sim N_3}
                        \sum_{N_1 \ll N_3} \int_{\R^{1+d}} (Q_1 \om_1^{-1}\til{u}_{N_1})(\ol{Q_2 \om_1^{-1}\til{v}_{N_2}}) 
                           \notag \\
            &\qquad \qquad       \cross (\ol{Q_{\ge M} \om \til{n}_{N_3}}) dxdt\Bigr{|}^2.                                   \label{k123}  
}
By the same manner as the estimate for Lemma \ref{tri} $(iv), i=4$, for $d=4, s=1/4$, we find  
\EQS{
 &\Bigl{|} \int_{\R^{1+4}} \Bigl{(}\sum_{N_1 \ll N_3}Q_1\om_1^{-1} \til{u}_{N_1}\Bigr{)}(\ol{Q_2 \om_1^{-1} \til{v}_{N_2}})
     (\ol{Q_{\ge M} \om \til{n}_{N_3}})dxdt\Bigr{|} \notag \\ 
 &\lec T^{1/4}\|u\|_{Y^{1/4}_{K_{\pm}}}\|v_{N_2}\|_{V^2_{K_{\pm}}}\|n_{N_3}\|_{V^2_{W_{\pm c}}}.                   \label{k123a}
}
$N_3 \gg 1, N_2 \sim N_3$ implies $N_2 \ge 1$. 
Hence, from \eqref{k123} and \eqref{k123a}, we have 
\EQQS{
 K_{1,2,3} 
  &\lec  \sum_{N_3 \gg 1} N_3^{1/2} \sup_{\|n\|_{V^2_{K_{\pm}}}=1} \Bigl{|}\sum_{N_2 \sim N_3} 
                \sum_{N_1 \ll N_3} \int_{\R^{1+4}} (Q_1 \om_1^{-1} \til{u}_{N_1})(\ol{Q_2 \om_1^{-1} \til{v}_{N_2}}) \\
  &\qquad \qquad    \cross (\ol{Q_{\ge M} \om \til{n}_{N_3}})dxdt\Bigr{|}^2 \\
  &\lec \sum_{N_2 \ge 1} N_2^{1/2}(T^{1/4}\|u\|_{Y^{1/4}_{K_{\pm}}} \|v_{N_2}\|_{V^2_{K_{\pm}}})^2  
    \lec T^{1/2} \|u\|_{Y^{1/4}_{K_{\pm}}}^2 \|v\|_{Y^{1/4}_{K_{\pm}}}^2.
}
By symmetry, the estimate for $K_2$ is obtained by the same manner as the estimate for $K_1$. 
Hence, we estimate $K_3$. 
By the triangle inequality, Lemma \ref{tri} $(i)$ and the Cauchy-Schwarz inequality, we have  
\EQS{
 K_3^{1/2} \lec \sum_{N_2} \sum_{N_1 \sim N_2}\Bigl{\{} \sum_{N_3 \lec N_2} N_3^{2s} \sup_{\|n\|_{V^2_{K_{\pm}}}=1}
                 \Bigl{|}\int_{\R^{1+d}} (\om_1^{-1}\til{u}_{N_1})(\ol{\om_1^{-1}\til{v}_{N_2}})
                    (\ol{\om \til{n}_{N_3}})dxdt\Bigr{|}^2 \Bigr{\}}^{1/2}. 
                                                                                                                                      \label{k3a}
}
If $d=4, s=1/4$, then we apply Lemma \ref{tri} $(i)$ and the Cauchy-Schwarz inequality, the right-hand side of 
\eqref{k3a} is bounded by  
\EQQS{
 &\sum_{N_2} \sum_{N_1 \sim N_2} \Bigl{\{} \sum_{N_3 \lec N_2} N_3^{1/2} (T^{1/4}N_3^{1/4}\|u_{N_1}\|_{V^2_{K_{\pm}}} 
     \|v_{N_2}\|_{V^2_{K_{\pm}}})^2 \Bigr{\}}^{1/2} \\
 &\lec T^{1/4} \sum_{N_2} \sum_{N_1 \sim N_2} (N_2 \|u_{N_1}\|_{V^2_{K_{\pm}}}^2 \|v_{N_2}\|_{V^2_{K_{\pm}}}^2 )^{1/2} \\
 &\lec T^{1/4} \Bigl{(}\sum_{N} N^{1/2} \|u_N\|_{V^2_{K_{\pm}}}^2\Bigr{)}^{1/2} 
             \Bigl{(}\sum_{N} N^{1/2} \|v_N\|_{V^2_{K_{\pm}}}^2\Bigr{)}^{1/2}.
}
Since 
\EQQS{
 \sum_{N < 1}N^{1/2}\|u_N\|_{V^2_{K_{\pm}}}^2 \lec \sum_{N < 1} N^{1/2} \|P_{<1}u\|_{V^2_{K_{\pm}}}^2 
                                                           \lec \|P_{<1}u\|_{V^2_{K_{\pm}}}^2, 
}
we obtain $K_3^{1/2} \lec T^{1/4} \|u\|_{Y^{1/4}_{K_{\pm}}} \|v\|_{Y^{1/4}_{K_{\pm}}}$.

Next, we prove \eqref{BEKG} for $d \ge 5$ and $s=s'=(d^2-3d-2)/2(d+1)$ by the same manner as the proof for $d=4, s=1/4$. 
From \eqref{j0} and Lemma \ref{tri} $(ii)$, we have 
\EQQS{
 J_0^{1/2} \lec T^{1/(d+1)}\|n\|_{\dot{Y}^{s'}_{W_{\pm c}}}\|v\|_{Y^{s'}_{K_{\pm}}}.  
}
By \eqref{j1}, $N_1 \sim N_2$, Lemma \ref{tri} $(iii)$ and 
$\|u_{N_1}\|_{V^2_{K_{\pm}}} \lec \|u\|_{V^2_{K_{\pm}}}$, we have 
\EQQS{
 J_1 \lec \sum_{N_2 \gec 1} N_2^{2s'} T^{2/(d+1)}\|n\|_{\dot{Y}^{s'}_{W_{\pm c}}}^2 \|v_{N_2}\|_{V^2_{K_{\pm}}}^2 
      \lec T^{2/(d+1)}\|n\|_{\dot{Y}^{s'}_{W_{\pm c}}}^2 \|v\|_{Y^{s'}_{K_{\pm}}}^2.
}
From Lemma \ref{tri} $(iv), N_3 \sim N_1 \ge 1$ and 
$\|u_{N_1}\|_{V^2_{K_{\pm}}} \lec \|u\|_{V^2_{K_{\pm}}}$, the right-hand side of \eqref{j2a} is bounded by  
\EQS{
 T^{2/(d+1)} \sum_{N_3 \gec 1} N_3^{2s'}  \|n_{N_3}\|_{V^2_{W_{\pm c}}}^2 \|v\|_{Y^{s'}_{K_{\pm}}}^2 
   \lec T^{2/(d+1)} \|n\|_{\dot{Y}^{s'}_{W_{\pm c}}}^2 \|v\|_{Y^{s'}_{K_{\pm}}}^2.                                         \label{al1}
}
From Lemma \ref{tri} $(iv), N_3 \sim N_1 \ge 1$ and 
$\|u_{N_1}\|_{V^2_{K_{\pm}}} \lec \|u\|_{V^2_{K_{\pm}}}$, the right-hand side of \eqref{j2b} is bounded by  
\EQS{
 T^{2/(d+1)} \sum_{N_3 \gec 1} N_3^{2s'}  \|n_{N_3}\|_{V^2_{W_{\pm c}}}^2 \|v\|_{Y^{s'}_{K_{\pm}}}^2               \label{al2}
   \lec T^{2/(d+1)} \|n\|_{\dot{Y}^{s'}_{W_{\pm c}}}^2 \|v\|_{Y^{s'}_{K_{\pm}}}^2.
}
From Lemma \ref{tri} $(iv), N_3 \sim N_1 \ge 1$ and 
$\|u_{N_1}\|_{V^2_{K_{\pm}}} \lec \|u\|_{V^2_{K_{\pm}}}$, the right-hand side of \eqref{j2c} is bounded by  
\EQS{
 T^{2/(d+1)} \sum_{N_3 \gec 1} N_3^{2s'}  \|n_{N_3}\|_{V^2_{W_{\pm c}}}^2 \|v\|_{Y^{s'}_{K_{\pm}}}^2 
   \lec T^{2/(d+1)} \|n\|_{\dot{Y}^{s'}_{W_{\pm c}}}^2 \|v\|_{Y^{s'}_{K_{\pm}}}^2.                                         \label{al3}
}
Collecting \eqref{al1}--\eqref{al3}, we obtain 
$J_2 \lec T^{2/(d+1)} \|n\|_{\dot{Y}^{s'}_{W_{\pm c}}}^2 \|v\|_{Y^{s'}_{K_{\pm}}}^2$.  
By the same manner as the estimate for Lemma \ref{tri} $(iii)$, we obtain 
\EQS{
 \Bigl{|}\int_{\R^{1+d}} \til{n}_{N_3} (\om_1^{-1}\til{v}_{N_2})\ol{\til{u}_{N_1}} dxdt\Bigr{|} 
    \lec T^{1/(d+1)} N_3^{s'} \|n_{N_3}\|_{V^2_{W_{\pm c}}} \|v_{N_2}\|_{V^2_{K_{\pm}}} \|u_{N_1}\|_{V^2_{K_{\pm}}}.   \label{j3c}
}
From \eqref{j3c}, the right-hand side of \eqref{j3a} is bounded by  
\EQQS{
 \sum_{N_1 \ge 1} \Bigl{(}\sum_{N_2 \gg N_1} \sum_{N_3 \sim N_2} N_1^{s'} T^{1/(d+1)} N_3^{s'} \|n_{N_3}\|_{V^2_{W_{\pm c}}} 
       \|v_{N_2}\|_{V^2_{K_{\pm}}}\Bigr{)}^2. 
}
Hence, $\| \cdot \|_{l^2 l^1} \lec \| \cdot \|_{l^1 l^2}$ and the Cauchy-Schwarz inequality to have 
\EQQS{
 J_3^{1/2} &\lec \sum_{N_2 \gec 1} \sum_{N_3 \sim N_2} \Bigl{(}\sum_{N_1 \ll N_2} N_1^{2s'} T^{2/(d+1)} N_3^{2s'} 
                        \|n_{N_3}\|_{V^2_{W_{\pm c}}}^2  \|v_{N_2}\|_{V^2_{K_{\pm}}}^2\Bigr{)}^{1/2} \\
              &\lec T^{1/(d+1)} \sum_{N_2 \gec 1} \sum_{N_3 \sim N_2} N_2^{2s'}N_3^{2s'}
                        \|n_{N_3}\|_{V^2_{W_{\pm c}}} \|v_{N_2}\|_{V^2_{K_{\pm}}} \\ 
              &\lec T^{1/(d+1)} \|n\|_{\dot{Y}^{s'}_{W_{\pm c}}} \|v\|_{Y^{s'}_{K_{\pm}}}. 
}
We prove \eqref{BEW} for $d \ge 5, s=s'=(d^2-3d-2)/2(d+1)$ by the same manner as the proof for $d=4, s=1/4$.  
By the H\"{o}lder inequality to have 
\EQS{
 &\Bigl{|}\int_{\R^{1+d}} \Bigl{(}\sum_{N_1 \ll N_3}\om_1^{-1}\til{u}_{N_1}\Bigr{)} 
            (\ol{\om_1^{-1}\til{v}_{N_2}})(\ol{\om \til{n}_{N_3}}) dxdt\Bigr{|} \notag \\
 &\lec \Bigl{\|}\sum_{N_1 \ll N_3}\om_1^{-1}\til{u}_{N_1}\Bigr{\|}_{L^{2(d+1)/(d-1)}_{t,x}} 
            \| \om_1^{-1}\til{v}_{N_2}\|_{L^{2(d+1)/(d-1)}_{t,x}} \| \om \til{n}_{N_3}\|_{L^{(d+1)/2}_{t,x}}.                        \label{k1a}
}
By Proposition \ref{Str}, $N_1 \ll N_3 \lec 1$ and discarding $\om_1^{-1}$ to have   
\EQS{
 \Bigl{\|}\sum_{N_1 \ll N_3}\om_1^{-1}\til{u}_{N_1}\Bigr{\|}_{L^{2(d+1)/(d-1)}_{t,x}} 
   &\lec \LR{N_1}^{1/2} \| \sum_{N_1 \ll N_3}\om_1^{-1}\til{u}_{N_1}\|_{V^2_{K_{\pm}}}  \notag \\
   &\lec \| \sum_{N_1 \ll N_3} \til{u}_{N_1}\|_{V^2_{K_{\pm}}}   \notag \\
   &\lec \|P_{<1}u\|_{V^2_{K_{\pm}}}.                                                                                                      \label{k1b}
}
From \eqref{K11}, \eqref{k1a}, \eqref{k1b}, \eqref{5dn2} and \eqref{5dn6}, we obtain  
\EQQS{
 K_{1,1} 
  &\lec \sum_{N_2 \lec 1} N_2^{2s'} (T^{1/(d+1)}\|P_{<1}u\|_{V^2_{K_{\pm}}} \|v_{N_2}\|_{V^2_{K_{\pm}}})^2 \\
  &\lec T^{2/(d+2)} \|P_{<1}u\|_{V^2_{K_{\pm}}}^2 \sum_{N_2 \lec 1} N_2^{2s'} \|v_{N_2}\|_{V^2_{K_{\pm}}}^2 \\
  &\lec T^{2/(d+2)} \|u\|_{Y^{s'}_{K_{\pm}}}^2 \|v\|_{Y^{s'}_{K_{\pm}}}^2. 
}   
By the same manner as the estimate for Lemma \ref{tri} $(iv), i=5$, we see 
\EQS{
 &\Bigl{|} \int_{\R^{1+d}} \Bigl{(}\sum_{N_1 \ll N_3}Q_{\ge M} \om_1^{-1} \til{u}_{N_1}\Bigr{)}(\ol{Q_2  \om_1^{-1} \til{v}_{N_2}})
     (\ol{Q_3 \om \til{n}_{N_3}})dxdt\Bigr{|} \notag \\ 
 &\lec T^{1/(d+1)}\|u\|_{Y^{s'}_{K_{\pm}}}\|v_{N_2}\|_{V^2_{K_{\pm}}}\|n_{N_3}\|_{V^2_{W_{\pm c}}}.                 \label{k121b}
}
From \eqref{K121a}, \eqref{k121} and \eqref{k121b}, we have 
\EQQS{
 K_{1,2,1} 
  &\lec  \sum_{N_3 \gg 1} N_3^{2s'} \sup_{\|n\|_{V^2_{K_{\pm}}}=1} \Bigl{|}\sum_{N_2 \sim N_3} 
                \sum_{N_1 \ll N_3} \int_{\R^{1+d}} (Q_{\ge M} \om_1^{-1} \til{u}_{N_1})(\ol{Q_2  \om_1^{-1} \til{v}_{N_2}}) \\
  &\qquad \qquad    \cross (\ol{Q_3 \om \til{n}_{N_3}})dxdt\Bigr{|}^2 \\
  &\lec \sum_{N_2 \ge 1} N_2^{2s'}(T^{1/(d+1)}\|u\|_{Y^{s'}_{K_{\pm}}} \|v_{N_2}\|_{V^2_{K_{\pm}}})^2  
    \lec T^{2/(d+1)} \|u\|_{Y^{s'}_{K_{\pm}}}^2 \|v\|_{Y^{s'}_{K_{\pm}}}^2.
}
By the same manner as the estimate for Lemma \ref{tri} $(iv), i=6$, we see 
\EQS{
 &\Bigl{|} \int_{\R^{1+d}} \Bigl{(}\sum_{N_1 \ll N_3}Q_1 \om_1^{-1} \til{u}_{N_1}\Bigr{)}(\ol{Q_{\ge M} \om_1^{-1} \til{v}_{N_2}})
     (\ol{Q_3 \om \til{n}_{N_3}})dxdt\Bigr{|} \notag \\ 
 &\lec T^{1/(d+1)}\|u\|_{Y^{s'}_{K_{\pm}}}\|v_{N_2}\|_{V^2_{K_{\pm}}}\|n_{N_3}\|_{V^2_{W_{\pm c}}}.                 \label{k122b}
}
From \eqref{K122a}, \eqref{k122} and \eqref{k122b}, we have 
\EQQS{
 K_{1,2,2} 
  &\lec  \sum_{N_3 \gg 1} N_3^{2s'} \sup_{\|n\|_{V^2_{K_{\pm}}}=1} \Bigl{|}\sum_{N_2 \sim N_3} 
                \sum_{N_1 \ll N_3} \int_{\R^{1+d}} (Q_1 \om_1^{-1} \til{u}_{N_1})(\ol{Q_{\ge M} \om_1^{-1} \til{v}_{N_2}}) \\
  &\qquad \qquad    \cross (\ol{Q_3 \om \til{n}_{N_3}})dxdt\Bigr{|}^2 \\
  &\lec \sum_{N_2 \ge 1} N_2^{2s'}(T^{1/(d+1)}\|u\|_{Y^{s'}_{K_{\pm}}} \|v_{N_2}\|_{V^2_{K_{\pm}}})^2  
    \lec T^{2/(d+1)} \|u\|_{Y^{s'}_{K_{\pm}}}^2 \|v\|_{Y^{s'}_{K_{\pm}}}^2.
}
By the same manner as the estimate for Lemma \ref{tri} $(iv), i=4$, we see 
\EQS{
 &\Bigl{|} \int_{\R^{1+d}} \Bigl{(}\sum_{N_1 \ll N_3}Q_1 \om_1^{-1} \til{u}_{N_1}\Bigr{)}(\ol{Q_2 \om_1^{-1} \til{v}_{N_2}})
     (\ol{Q_{\ge M} \om \til{n}_{N_3}})dxdt\Bigr{|} \notag \\ 
 &\lec T^{1/(d+1)}\|u\|_{Y^{s'}_{K_{\pm}}}\|v_{N_2}\|_{V^2_{K_{\pm}}}\|n_{N_3}\|_{V^2_{W_{\pm c}}}.                 \label{k123b}
}
From \eqref{K123a}, \eqref{k123} and \eqref{k123b}, we have 
\EQQS{
 K_{1,2,3} 
  &\lec \sum_{N_3 \gg 1} N_3^{2s'} \sup_{\|n\|_{V^2_{K_{\pm}}}=1} \Bigl{|}\sum_{N_2 \sim N_3} 
                \sum_{N_1 \ll N_3} \int_{\R^{1+d}} (Q_1 \om_1^{-1} \til{u}_{N_1})(\ol{Q_2 \om_1^{-1} \til{v}_{N_2}}) \\
  &\qquad \qquad    \cross (\ol{Q_{\ge M} \om \til{n}_{N_3}})dxdt\Bigr{|}^2 \\
  &\lec \sum_{N_2 \ge 1} N_2^{2s'}(T^{1/(d+1)}\|u\|_{Y^{s'}_{K_{\pm}}} \|v_{N_2}\|_{V^2_{K_{\pm}}})^2  
    \lec T^{2/(d+1)} \|u\|_{Y^{s'}_{K_{\pm}}}^2 \|v\|_{Y^{s'}_{K_{\pm}}}^2.
}
By symmetry, the estimate for $K_2$ is obtained by the same manner as the estimate for $K_1$. 
We apply Lemma \ref{tri} $(i)$ and the Cauchy-Schwarz inequality, 
the right-hand side of \eqref{k3a} is bounded by   
\EQQS{
 &\sum_{N_2} \sum_{N_1 \sim N_2} \Bigl{\{} \sum_{N_3 \lec N_2} N_3^{2s'} (T^{1/(d+1)}N_3^{s'}\|u_{N_1}\|_{V^2_{K_{\pm}}} 
     \|v_{N_2}\|_{V^2_{K_{\pm}}})^2 \Bigr{\}}^{1/2} \\
 &\lec T^{1/(d+1)} \sum_{N_2} \sum_{N_1 \sim N_2} (N_2^{4s'} \|u_{N_1}\|_{V^2_{K_{\pm}}}^2 
               \|v_{N_2}\|_{V^2_{K_{\pm}}}^2 )^{1/2} \\
 &\lec T^{1/(d+1)} \Bigl{(}\sum_{N} N^{2s'} \|u_N\|_{V^2_{K_{\pm}}}^2\Bigr{)}^{1/2} 
               \Bigl{(}\sum_{N} N^{2s'} \|v_N\|_{V^2_{K_{\pm}}}^2\Bigr{)}^{1/2}.
}
Since $s' > 0$, we have 
\EQQS{
 \sum_{N < 1}N^{2s'}\|u_N\|_{V^2_{K_{\pm}}}^2 \lec \sum_{N < 1} N^{2s'} \|P_{<1}u\|_{V^2_{K_{\pm}}}^2 
                                                          \lec \|P_{<1}u\|_{V^2_{K_{\pm}}}^2. 
}
Thus, we obtain $K_3^{1/2} \lec T^{1/(d+1)} \|u\|_{Y^{s'}_{K_{\pm}}} \|v\|_{Y^{s'}_{K_{\pm}}}$.

Finally, we prove \eqref{BEKG} for $d \ge 4, s=s_c=d/2-2$ and spherically symmetric functions $(u,v,n)$ 
by the same manner as the proof of $d=4, s=1/4$. 
From \eqref{j0} and Lemma \ref{tri} $(ii)$, we obtain 
\EQQS{
 J_0^{1/2} \lec \|n\|_{\dot{Y}^{s_c}_{W_{\pm c}}}\|v\|_{Y^{s_c}_{K_{\pm}}}.
}
By \eqref{j1}, $N_1 \sim N_2$, Lemma \ref{tri} $(iii)$ and $\|u_{N_1}\|_{V^2_{K_{\pm}}} \lec \|u\|_{V^2_{K_{\pm}}}$, we have    
\EQQS{
 J_1 \lec \sum_{N_2 \gec 1} N_2^{2s_c} \|n\|_{\dot{Y}^{s_c}_{W_{\pm c}}}^2 \|v_{N_2}\|_{V^2_{K_{\pm}}}^2 
      \lec \|n\|_{\dot{Y}^{s_c}_{W_{\pm c}}}^2 \|v\|_{Y^{s_c}_{K_{\pm}}}^2. 
} 
From Lemma \ref{tri} $(iv), N_3 \sim N_1 \ge 1$ and $\|u_{N_1}\|_{V^2_{K_{\pm}}} \lec \|u\|_{V^2_{K_{\pm}}}$, 
the right-hand side of \eqref{j2a} is bounded by  
\EQS{
 \sum_{N_3 \gec 1} N_3^{2s_c}  \|n_{N_3}\|_{V^2_{W_{\pm c}}}^2 \|v\|_{Y^{s_c}_{K_{\pm}}}^2 
   \lec \|n\|_{\dot{Y}^{s_c}_{W_{\pm c}}}^2 \|v\|_{Y^{s_c}_{K_{\pm}}}^2.                                                      \label{Jj21}
} 
From Lemma \ref{tri} $(iv), N_3 \sim N_1 \ge 1$ and $\|u_{N_1}\|_{V^2_{K_{\pm}}} \lec \|u\|_{V^2_{K_{\pm}}}$, 
the right-hand side of \eqref{j2b} is bounded by  
\EQS{
 \sum_{N_3 \gec 1} N_3^{2s_c}  \|n_{N_3}\|_{V^2_{W_{\pm c}}}^2 \|v\|_{Y^{s_c}_{K_{\pm}}}^2 
   \lec \|n\|_{\dot{Y}^{s_c}_{W_{\pm c}}}^2 \|v\|_{Y^{s_c}_{K_{\pm}}}^2.                                                      \label{Jj22}
}  
From Lemma \ref{tri} $(iv), N_3 \sim N_1 \ge 1$ and $\|u_{N_1}\|_{V^2_{K_{\pm}}} \lec \|u\|_{V^2_{K_{\pm}}}$, 
the right-hand side of \eqref{j2c} is bounded by  
\EQS{
 \sum_{N_3 \gec 1} N_3^{2s_c}  \|n_{N_3}\|_{V^2_{W_{\pm c}}}^2 \|v\|_{Y^{s_c}_{K_{\pm}}}^2 
   \lec \|n\|_{\dot{Y}^{s_c}_{W_{\pm c}}}^2 \|v\|_{Y^{s_c}_{K_{\pm}}}^2.                                                      \label{Jj23}
} 
Collecting \eqref{Jj21}--\eqref{Jj23}, we have 
$J_2 \lec \|n\|_{\dot{Y}^{s_c}_{W_{\pm c}}}^2 \|v\|_{Y^{s_c}_{K_{\pm}}}^2$.  
By the same manner as the estimate for Lemma \ref{tri} $(iii)$, we obtain 
\EQS{
 \Bigl{|}\int_{\R^{1+d}} \til{n}_{N_3} (\om_1^{-1}\til{v}_{N_2})\ol{\til{u}_{N_1}} dxdt\Bigr{|} 
    \lec N_3^{s_c} \|n_{N_3}\|_{V^2_{W_{\pm c}}} \|v_{N_2}\|_{V^2_{K_{\pm}}} \|u_{N_1}\|_{V^2_{K_{\pm}}}.   \label{j3d}
}
From \eqref{j3d}, the right-hand side of \eqref{j3a} is bounded by  
\EQQS{
 \sum_{N_1 \ge 1} \Bigl{(}\sum_{N_2 \gg N_1} \sum_{N_3 \sim N_2} N_1^{s_c} N_3^{s_c} \|n_{N_3}\|_{V^2_{W_{\pm c}}} 
       \|v_{N_2}\|_{V^2_{K_{\pm}}}\Bigr{)}^2. 
}
Hence, $\| \cdot \|_{l^2 l^1} \lec \| \cdot \|_{l^1 l^2}$ and the Cauchy-Schwarz inequality to have 
\EQQS{
 J_3^{1/2} &\lec \sum_{N_2 \gec 1} \sum_{N_3 \sim N_2} \Bigl{(}\sum_{N_1 \ll N_2} N_1^{2s_c} N_3^{2s_c} 
                        \|n_{N_3}\|_{V^2_{W_{\pm c}}}^2  \|v_{N_2}\|_{V^2_{K_{\pm}}}^2\Bigr{)}^{1/2} \\
              &\lec \sum_{N_2 \gec 1} \sum_{N_3 \sim N_2} N_2^{2s_c}N_3^{2s_c}
                        \|n_{N_3}\|_{V^2_{W_{\pm c}}} \|v_{N_2}\|_{V^2_{K_{\pm}}} \\ 
              &\lec \|n\|_{\dot{Y}^{s_c}_{W_{\pm c}}} \|v\|_{Y^{s_c}_{K_{\pm}}}. 
}
We prove \eqref{BEW} for $d \ge 4, s=s_c=d/2-2$ and spherically symmetric functions $(u,v,n)$ 
by the same manner as the proof of $d=4, s=1/4$.  
By the H\"{o}lder inequality to have  
\EQS{
 &\Bigl{|}\int_{\R^{1+d}} \Bigl{(}\sum_{N_1 \ll N_3}\om_1^{-1}\til{u}_{N_1}\Bigr{)} (\ol{\om_1^{-1}\til{v}_{N_2}})
            (\ol{\om \til{n}_{N_3}}) dxdt\Bigr{|} \notag \\
 &\lec \Bigl{\|} \sum_{N_1 \ll N_3} \om_1^{-1}\til{u}_{N_1}\Bigr{\|}_{L^3_{t,x}} \|\om_1^{-1}\til{v}_{N_2}\|_{L^3_{t,x}} 
            \|\om \til{n}_{N_3}\|_{L^3_{t,x}}.                                                                                                \label{k1c} 
}
Discarding $\om_1^{-1}$, then $N_1 \ll N_3 \lec 1$ and the same manner as \eqref{4dr2}, we find 
\EQS{
 \Bigl{\|} \sum_{N_1 \ll N_3} \om_1^{-1}\til{u}_{N_1}\Bigr{\|}_{L^3_{t,x}} 
   \lec \LR{N_1}^{(d-2)/6} \Bigl{\|} \sum_{N_1 \ll N_3} \til{u}_{N_1}\Bigr{\|}_{V^2_{K_{\pm}}} 
   \lec \|P_{<1}u\|_{V^2_{K_{\pm}}}.                                                                                                       \label{k1d}  
}
Collecting \eqref{K11}, \eqref{k1c}, \eqref{k1d}, \eqref{4dr3}, \eqref{4dr4} and $N_2 \sim N_3 \lec 1$, we obtain 
\EQQS{
 K_{1,1} &\lec \sum_{N_2 \lec 1} N_2^{2s_c} (\|P_{<1}u\|_{V^2_{K_{\pm}}} \LR{N_2}^{(d-8)/6}\|v_{N_2}\|_{V^2_{K_{\pm}}}
                   N_2^{(d+4)/6})^2 \\
          &\lec \|P_{<1}u\|_{V^2_{K_{\pm}}}^2 \sum_{N_2 \lec 1} N_2^{2s_c}\|v_{N_2}\|_{V^2_{K_{\pm}}}^2 \\
          &\lec \|u\|_{Y^{s_c}_{K_{\pm}}}^2 \|v\|_{Y^{s_c}_{K_{\pm}}}^2. 
}
By the same manner as the estimate for Lemma \ref{tri} $(iv), i=5$, we obtain 
\EQS{
 &\Bigl{|} \int_{\R^{1+d}} \Bigl{(}\sum_{N_1 \ll N_3}Q_{\ge M} \om_1^{-1} \til{u}_{N_1}\Bigr{)}(\ol{Q_2  \om_1^{-1} \til{v}_{N_2}})
     (\ol{Q_3 \om \til{n}_{N_3}})dxdt\Bigr{|} \notag \\ 
 &\lec \|u\|_{Y^{s_c}_{K_{\pm}}}\|v_{N_2}\|_{V^2_{K_{\pm}}}\|n_{N_3}\|_{V^2_{W_{\pm c}}}.                              \label{k121c}
}
From \eqref{K121a}, \eqref{k121} and \eqref{k121c}, we have 
\EQQS{
 K_{1,2,1} 
  &\lec  \sum_{N_3 \gg 1} N_3^{2s_c} \sup_{\|n\|_{V^2_{K_{\pm}}}=1} \Bigl{|}\sum_{N_2 \sim N_3} 
                \sum_{N_1 \ll N_3} \int_{\R^{1+d}} (Q_{\ge M} \om_1^{-1} \til{u}_{N_1})(\ol{Q_2  \om_1^{-1} \til{v}_{N_2}}) \\
  &\qquad \qquad    \cross (\ol{Q_3 \om \til{n}_{N_3}})dxdt\Bigr{|}^2 \\
  &\lec \sum_{N_2 \ge 1} N_2^{2s_c}(\|u\|_{Y^{s_c}_{K_{\pm}}} \|v_{N_2}\|_{V^2_{K_{\pm}}})^2  
    \lec \|u\|_{Y^{s_c}_{K_{\pm}}}^2 \|v\|_{Y^{s_c}_{K_{\pm}}}^2.
}
By the same manner as the estimate for Lemma \ref{tri} $(iv), i=6$, we obtain 
\EQS{
 &\Bigl{|} \int_{\R^{1+d}} \Bigl{(}\sum_{N_1 \ll N_3}Q_1 \om_1^{-1} \til{u}_{N_1}\Bigr{)}(\ol{Q_{\ge M} \om_1^{-1} \til{v}_{N_2}})
     (\ol{Q_3 \om \til{n}_{N_3}})dxdt\Bigr{|} \notag \\ 
 &\lec \|u\|_{Y^{s_c}_{K_{\pm}}}\|v_{N_2}\|_{V^2_{K_{\pm}}}\|n_{N_3}\|_{V^2_{W_{\pm c}}}.                              \label{k122c}
}
From \eqref{K122a}, \eqref{k122} and \eqref{k122c}, we have 
\EQQS{
 K_{1,2,2} 
  &\lec  \sum_{N_3 \gg 1} N_3^{2s_c} \sup_{\|n\|_{V^2_{K_{\pm}}}=1} \Bigl{|}\sum_{N_2 \sim N_3} 
                \sum_{N_1 \ll N_3} \int_{\R^{1+d}} (Q_1 \om_1^{-1} \til{u}_{N_1})(\ol{Q_{\ge M} \om_1^{-1} \til{v}_{N_2}}) \\
  &\qquad \qquad    \cross (\ol{Q_3 \om \til{n}_{N_3}})dxdt\Bigr{|}^2 \\
  &\lec \sum_{N_2 \ge 1} N_2^{2s_c}(\|u\|_{Y^{s_c}_{K_{\pm}}} \|v_{N_2}\|_{V^2_{K_{\pm}}})^2  
    \lec \|u\|_{Y^{s_c}_{K_{\pm}}}^2 \|v\|_{Y^{s_c}_{K_{\pm}}}^2.
}
By the same manner as the estimate for Lemma \ref{tri} $(iv), i=4$, we obtain 
\EQS{
 &\Bigl{|} \int_{\R^{1+d}} \Bigl{(}\sum_{N_1 \ll N_3}Q_1 \om_1^{-1} \til{u}_{N_1}\Bigr{)}(\ol{Q_2 \om_1^{-1} \til{v}_{N_2}})
     (\ol{Q_{\ge M} \om \til{n}_{N_3}})dxdt\Bigr{|} \notag \\ 
 &\lec \|u\|_{Y^{s_c}_{K_{\pm}}}\|v_{N_2}\|_{V^2_{K_{\pm}}}\|n_{N_3}\|_{V^2_{W_{\pm c}}}.                              \label{k123c}
}
From \eqref{K123a}, \eqref{k123} and \eqref{k123c}, we have 
\EQQS{
 K_{1,2,3} 
  &\lec  \sum_{N_3 \gg 1} N_3^{2s_c} \sup_{\|n\|_{V^2_{K_{\pm}}}=1} \Bigl{|}\sum_{N_2 \sim N_3} 
                \sum_{N_1 \ll N_3} \int_{\R^{1+d}} (Q_1 \om_1^{-1} \til{u}_{N_1})(\ol{Q_2 \om_1^{-1} \til{v}_{N_2}}) \\
  &\qquad \qquad    \cross (\ol{Q_{\ge M} \om \til{n}_{N_3}})dxdt\Bigr{|}^2 \\
  &\lec \sum_{N_2 \ge 1} N_2^{2s_c}(\|u\|_{Y^{s_c}_{K_{\pm}}} \|v_{N_2}\|_{V^2_{K_{\pm}}})^2  
    \lec \|u\|_{Y^{s_c}_{K_{\pm}}}^2 \|v\|_{Y^{s_c}_{K_{\pm}}}^2.
}
By symmetry, the estimate for $K_2$ is obtained by the same manner as the estimate for $K_1$. 
For $d = 4$, from \eqref{k3a}, Lemma \ref{tri} $(i)$ and the Cauchy-Schwarz inequality to have 
\EQQS{
 K_3^{1/2} 
  &\lec \sum_{N_2} \sum_{N_1 \sim N_2} \Bigl{\{} \sum_{N_3 \lec N_2} (\LR{N_2}^{-4/3}N_3^{4/3}\|u_{N_1}\|_{V^2_{K_{\pm}}} 
            \|v_{N_2}\|_{V^2_{K_{\pm}}})^2 \Bigr{\}}^{1/2} \\
  &\lec \sum_{N_2} \sum_{N_1 \sim N_2} \bigl{(}\LR{N_2}^{-8/3}N_2^{8/3} \|u_{N_1}\|_{V^2_{K_{\pm}}}^2 
            \|v_{N_2}\|_{V^2_{K_{\pm}}}^2 \bigr{)}^{1/2} \\
  &\lec \Bigl{(}\sum_{N} \LR{N}^{-4/3}N^{4/3} \|u_N\|_{V^2_{K_{\pm}}}^2\Bigr{)}^{1/2} 
            \Bigl{(}\sum_{N} \LR{N}^{-4/3} N^{4/3} \|v_N\|_{V^2_{K_{\pm}}}^2\Bigr{)}^{1/2}.
}
By $\LR{N}^{-4/3} \le 1$, we have 
\EQQS{
 \sum_{N < 1} \LR{N}^{-4/3} N^{4/3}\|u_N\|_{V^2_{K_{\pm}}}^2 
   \lec \sum_{N < 1} N^{4/3} \|P_{<1}u\|_{V^2_{K_{\pm}}}^2 
   \lec \|P_{<1}u\|_{V^2_{K_{\pm}}}^2.  
}
Hence, for $d=4$, we obtain 
\EQS{
 K_3^{1/2} \lec \|u\|_{Y^0_{K_{\pm}}} \|v\|_{Y^0_{K_{\pm}}}.        \label{k3r4}
}  
For $d > 4$, from \eqref{k3a} and Lemma \ref{tri} $(i)$, we have 
\EQS{
 K_3^{1/2} 
  &\lec \sum_{N_2} \sum_{N_1 \sim N_2} \Bigl{\{} \sum_{N_3 \lec N_2} N_3^{2s_c} (\LR{N_2}^{(d-8)/3}N_3^{(d+4)/6}
            \|u_{N_1}\|_{V^2_{K_{\pm}}} \|v_{N_2}\|_{V^2_{K_{\pm}}})^2 \Bigr{\}}^{1/2} \notag \\
  &\lec \sum_{N_2}\sum_{N_1 \sim N_2} \Bigl{(} \sum_{N_3 \lec N_2} N_3^{4(d-2)/3}\LR{N_2}^{2(d-8)/3}
              \|u_{N_1}\|_{V^2_{K_{\pm}}}^2\|v_{N_2}\|_{V^2_{K_{\pm}}}^2\Bigr{)}^{1/2} \notag \\
  &\lec \sum_{N_2}\sum_{N_1 \sim N_2} N_2^{2(d-2)/3}\LR{N_2}^{(d-8)/3}
              \|u_{N_1}\|_{V^2_{K_{\pm}}}\|v_{N_2}\|_{V^2_{K_{\pm}}}.                                                          \label{K3r1}
}
For $d \le 8$, then $\LR{N_2}^{(d-8)/3} \le N_2^{(d-8)/3}$. 
Hence by \eqref{K3r1} and the Cauchy-Schwarz inequality, for $4 < d \le 8$, we have  
\EQS{
 K_3^{1/2} &\lec \sum_{N_2}\sum_{N_1 \sim N_2} N_2^{d-4}
                        \|u_{N_1}\|_{V^2_{K_{\pm}}}\|v_{N_2}\|_{V^2_{K_{\pm}}} \notag \\
              &\lec \|u\|_{Y^{s_c}_{K_{\pm}}} \|v\|_{Y^{s_c}_{K_{\pm}}}.   \label{k3r48}
}  
For $d > 8$ and $N_2 < 1$, it holds that $\LR{N_2} \lec 1$. 
Hence, by \eqref{K3r1} to have 
\EQS{
 K_3^{1/2} &\lec \sum_{N_2 < 1} \sum_{N_1 \sim N_2} N_2^{2(d-2)/3}\|u_{N_1}\|_{V^2_{K_{\pm}}} \|v_{N_2}\|_{V^2_{K_{\pm}}} 
                             \notag \\
              &\lec \|u\|_{Y^{s_c}_{K_{\pm}}}\|P_{<1}v\|_{V^2_{K_{\pm}}}.                                                \label{K3r8l1}
}   
For $d > 8$ and $N_2 \ge 1$, it holds that $\LR{N_2}^{(d-8)/3} \sim N_2^{(d-8)/3}$. 
Thus by \eqref{k3a} and the Cauchy-Schwarz inequality to have 
\EQS{
 K_3^{1/2} &\lec \sum_{N_2}\sum_{N_1 \sim N_2} N_2^{d-4}
                        \|u_{N_1}\|_{V^2_{K_{\pm}}}\|v_{N_2}\|_{V^2_{K_{\pm}}} \notag \\
              &\lec \|u\|_{Y^{s_c}_{K_{\pm}}} \|v\|_{Y^{s_c}_{K_{\pm}}}.                                   \label{K3r8g1}
}
Collecting \eqref{k3r4}, \eqref{k3r48}--\eqref{K3r8g1}, we obtain 
$K_3^{1/2} \lec \|u\|_{Y^{s_c}_{K_{\pm}}} \|v\|_{Y^{s_c}_{K_{\pm}}}$ for $d \ge 4$.

\end{proof}

\section{The proof of the main theorem}
We define 
\EQQS{
 u_{\pm} := \om_1 u \pm i\p_t u, \quad  
 n_{\pm} := n \pm i (c\om)^{-1} \p_t n
}
where $\om_1 := (1-\laplacian )^{1/2}, \om := (-\laplacian )^{1/2}$.  
Then the wave equation in \eqref{KGZ} is rewritten into
\EQS{
 \begin{cases}
  i\p_t u_{\pm} \mp \om_1 u_{\pm} 
    = \pm (1/4)(n_+ + n_-)(\om_1^{-1}u_+ + \om_1^{-1}u_-), 
               \qquad (t,x) \in [-T,T] \cross \R^d, \\
  i\p_t n_{\pm} \mp c\om n_{\pm} 
    = \pm (4c)^{-1}\om | \om_1^{-1} u_+ + \om_1^{-1} u_-|^2, \qquad (t,x) \in [-T,T] \cross \R^d, \\
  (u_{\pm}, n_{\pm})|_{t=0} = (u_{\pm 0}, n_{\pm 0}) 
                                \in H^s(\R^d) \cross \dot{H}^s(\R^d). 
                                                                                                                                            \label{KGZ''}
 \end{cases}  
}

Hence by the Duhamel principle, 
we consider the following integral equation corresponding to \eqref{KGZ''} on the time interval 
$[0, T)$ with $0< T \le \I:$
\EQ{ \label{int-eq}
  u_{\pm}=\Phi_1(u_{\pm},n_+,n_-), \quad  n_{\pm}=\Phi_2(n_{\pm},u_+,u_-), 
}
where
\EQQS{
  &\Phi_1(u_{\pm},n_+,n_-) := K_{\pm}(t)u_{\pm0} \pm (1/4)\{ I_{T,K_{\pm}}(n_+,u_+)(t) + I_{T,K_{\pm}}(n_+,u_-)(t) \\
                                 &\qquad \qquad \qquad \qquad  
                                      + I_{T,K_{\pm}}(n_-,u_+)(t) + I_{T,K_{\pm}}(n_-,u_-)(t) \}, \\
  &\Phi_2(n_{\pm},u_+,u_-) := W_{\pm c}(t)n_{\pm 0} \pm (4c)^{-1}\{ I_{T,W_{\pm c}}(u_+,u_+)(t) + I_{T,W_{\pm c}}(u_+,u_-)(t) \\
                                 &\qquad \qquad \qquad \qquad 
                                      + I_{T,W_{\pm c}}(u_-,u_+)(t) + I_{T,W_{\pm c}}(u_-,u_-)(t)\}.
}
\begin{prop}\label{main_prop1}
 (i) Let $s=1/4$ for $d = 4$ or $s=(d^2-3d-2)/2(d+1)$ for $d \ge 5$. 
     Let $\de > 0$ be arbitrary. 
     Then, for any initial data 
    $(u_{\pm0}, n_{\pm 0}) \in B_{\de }(H^s(\R^d) \cross \dot{H}^s(\R^d))$, 
     there exists $T > 0$ and a unique solution of \eqref{int-eq} on $[0,T]$ such that 
\EQQ{
 (u_{\pm}, n_{\pm}) \in Y^s_{K_{\pm}}([0, T]) \cross \dot{Y}^s_{W_{\pm c}}([0, T])
      \subset C([0, T]; H^s(\R^d)) \cross C([0, T]; \dot{H}^s(\R^d)). 
} 
Moreover, let $d \ge 4, s=s_c=d/2-2$ and $\de > 0$ be sufficiently small. 
If $(u_{\pm0}, n_{\pm 0}) \in B_{\de }(H^s(\R^d) \cross \dot{H}^s(\R^d))$ be radial, 
then for all $0 < T < \I$, there exists a unique spherically symmetric solution of \eqref{int-eq} on $[0,T]$ such that 
\EQQ{
 (u_{\pm}, n_{\pm}) \in Y^s_{K_{\pm}}([0, T]) \cross \dot{Y}^s_{W_{\pm c}}([0, T])
      \subset C([0, T]; H^s(\R^d)) \cross C([0, T]; \dot{H}^s(\R^d)). 
}  
 (ii) The flow map obtained by (i):\\
    $B_{\de }(H^s(\R^d)) \cross B_{\de }(\dot{H}^s(\R^d)) 
 \ni (u_{\pm 0}, n_{\pm 0}) \mapsto
           (u_{\pm}, n_{\pm}) \in Y^s_{K_{\pm}}([0, T]) \cross \dot{Y}^s_{W_{\pm c}}([0, T])$
    is Lipschitz continuous.  
\end{prop}
\begin{rem} \label{time_rev}
Due to the time reversibility of the Klein-Gordon-Zakharov equation, 
Porpositions \ref{main_prop1} also holds in corresponding time interval $[-T,0]$
\end{rem} 
\begin{rem} \label{gsol}
By $(i)$ in Proposition \ref{main_prop1} and Remark \ref{time_rev}, for any $T>0$, we have solutions to 
\eqref{int-eq} $(u_{\pm}(t),n_{\pm}(t))$ on $[0,T]$ and $[-T,0]$. 
If radial initial data $(u_{\pm0}, n_{\pm 0}) \in B_{\de }(H^s(\R^d) \cross \dot{H}^s(\R^d))$, 
then we can take $T$ arbitrary large and by uniqueness, spherically symmetric function   
$(u_{\pm}(t),n_{\pm}(t)) \in C((-\I,\I);H^s(\R^d)) \cross C((-\I,\I);\dot{H}^s(\R^d))$ can be defined uniquely. 
\end{rem} 
\begin{prop} \label{main_prop2}
 Let the spherically symmetric solution $(u_{\pm}(t), n_{\pm}(t))$ to \eqref{int-eq} on $(-\I,\I)$ obtained by 
Proposirion \ref{main_prop1}, Remark \ref{time_rev} and Remark \ref{gsol} with radial initial data 
$(u_{\pm 0}, n_{\pm 0}) \in B_{\de}(H^s(\R^d) \cross \dot{H}^s(\R^d))$. 
Then, there exist $(u_{\pm, +\I}, n_{\pm, +\I})$ and $(u_{\pm, -\I}, n_{\pm, -\I})$ in $H^s(\R^d) \cross \dot{H}^s(\R^d)$ 
such that 
\EQQS{
 \|u_{\pm}(t)-K_{\pm}(t)u_{\pm, +\I}\|_{H^s_x(\R^d)} + \|n_{\pm}(t)-W_{\pm c}(t)n_{\pm, +\I}\|_{\dot{H}^s_x(\R^d)} \to 0
}  
as $t \to +\I$ and 
\EQQS{
 \|u_{\pm}(t)-K_{\pm}(t)u_{\pm, -\I}\|_{H^s_x(\R^d)} + \|n_{\pm}(t)-W_{\pm c}(t)n_{\pm, -\I}\|_{\dot{H}^s_x(\R^d)} \to 0
}  
as $t \to -\I$.  
\end{prop}

\begin{proof}[proof of Proposition \ref{main_prop1}]
First, we prove $(i)$. 
By Proposition \ref{Str}, there exists $C>0$ such that   
\EQQS{
 \|K_{\pm}(t)u_{\pm0}\|_{Y^s_{K_{\pm}}} \le C\|u_{\pm0}\|_{H^s}, \qquad 
 \|W_{\pm c}(t)n_{\pm 0}\|_{\dot{Y}^s_{W_{\pm c}}} \le C\|n_{\pm 0}\|_{\dot{H}^s}.
} 
We denote time interval $I := [0, T]$. 
If $(u_{\pm 0}, n_{\pm 0}) \in B_{\de}(H^s(\R^d) \cross \dot{H}^s(\R^d)), 
(u_{\pm}, n_{\pm}) \in B_r(Y^s_{K_{\pm}}(I) \cross \dot{Y}^s_{W_{\pm c}}(I))$,  
then by Proposition \ref{BE}, for $(\th, s)=(1/4,1/4), d=4$ or for $(\th, s)=(1/(d+1), (d^2-3d-2)/2(d+1)), d \ge 5$, it holds that   
\EQQS{
 &\|\Phi_1(u_{\pm}, n_+,n_-)\|_{Y^s_{K_{\pm}}(I)} \\
  &\le C\|u_{\pm0}\|_{H^s}+(1/4)CT^{\th}(\|n_+\|_{\dot{Y}^s_{W_{+c}}(I)}\|u_+\|_{Y^s_{K_+}(I)} 
         + \|n_+\|_{\dot{Y}^s_{W_{+c}}(I)}\|u_-\|_{Y^s_{K_-}(I)} \\
  &\qquad 
         + \|n_-\|_{\dot{Y}^s_{W_{-c}}(I)}\|u_+\|_{Y^s_{K_+}(I)} 
         + \|n_-\|_{\dot{Y}^s_{W_{-c}}(I)}\|u_-\|_{Y^s_{K_-}(I)}) \\
  &\le C\de + CT^{\th}r^2, \\
 &\|\Phi_2(n_{\pm},u_+,u_-)\|_{\dot{Y}^s_{W_{\pm c}}(I)} \\
  &\le C\|n_{\pm 0}\|_{H^s}+(CT^{\th}/4c)(\|u_+\|^2_{Y^s_{K_+}(I)} + 2\|u_+\|_{Y^s_{K_+}(I)}\|u_-\|_{Y^s_{K_-}(I)} 
         + \|u_-\|_{Y^s_{K_-}(I)}^2) \\
  &\le C\de +CT^{\th}r^2/c.
}
We take $r=2C\de$ and $T > 0$  satisfying 
\EQS{
 4CT^{\th}r \le \min\{1,c\}.        \label{T}
}
Then we have   
\EQQS{
\|\Phi_1(u_{\pm}, n_+,n_-)\|_{Y^s_{K_{\pm}}(I)} \le r,\qquad \|\Phi_2(n_{\pm},u_+,u_-)\|_{\dot{Y}^s_{W_{\pm c}}(I)} \le r.
}
Hence, $(\Phi_1, \Phi_2)$ is a map from $B_r(Y^s_{K_{\pm}}([0,T]) \cross \dot{Y}^s_{W_{\pm c}}([0,T]))$ into itself.  
Similarly, we assume $(v_{\pm0}, m_{\pm 0}) \in B_{\de}(H^s(\R^d) \cross \dot{H}^s(\R^d)), 
(v_{\pm}, m_{\pm}) \in B_r(Y^s_{K_{\pm}}(I) \cross \dot{Y}^s_{W_{\pm c}}(I))$, then it holds that 
\EQS{
 &\|\Phi_1(u_{\pm}, n_+, n_-)-\Phi_1(v_{\pm}, m_+,m_-)\|_{Y^s_{K_{\pm}}(I)} \notag \\
  &\le (1/4)(\|I_{T,K_{\pm}}(n_+, u_+)(t)-I_{T,K_{\pm}}(m_+, v_+)(t)\|_{Y^s_{K_{\pm}}(I)} \notag \\
  &\qquad \qquad + \|I_{T,K_{\pm}}(n_+, u_-)(t)-I_{T,K_{\pm}}(m_+, v_-)(t)\|_{Y^s_{K_{\pm}}(I)} \notag \\
  &\qquad \qquad + \|I_{T,K_{\pm}}(n_-, u_+)(t)-I_{T,K_{\pm}}(m_-, v_+)(t)\|_{Y^s_{K_{\pm}}(I)} \notag \\
  &\qquad \qquad + \|I_{T,K_{\pm}}(n_-, u_-)(t)-I_{T,K_{\pm}}(m_-, v_-)(t)\|_{Y^s_{K_{\pm}}(I)}).  \label{con}
}
By Proposition \ref{BE}, we have 
\EQS{
 &\|I_{T,K_{\pm}}(n_+, u_+)(t)-I_{T,K_{\pm}}(m_+, v_+)(t)\|_{Y^s_{K_{\pm}}(I)} \notag \\
   &\le CT^{\th}(\|n_+-m_+\|_{\dot{Y}^s_{W_{+c}}(I)}\|u_+\|_{Y^s_{K_+}(I)} + \|m_+\|_{Y^s_{W_{+c}}(I)}\|u_+ - v_+\|_{Y^s_{K_+}(I)}).      
              \label{pp}
}
Similarly, we have 
\EQS{
 &\|I_{T,K_{\pm}}(n_+, u_-)(t)-I_{T,K_{\pm}}(m_+, v_-)(t)\|_{Y^s_{K_{\pm}}(I)} \notag \\
   &\le CT^{\th}(\|n_+ - m_+\|_{\dot{Y}^s_{W_{+c}}(I)}\|u_-\|_{Y^s_{K_-}(I)} + \|m_+\|_{\dot{Y}^s_{W_{+c}}(I)} 
              \|u_- - v_-\|_{Y^s_{K_-}(I)}),  \label{pn} \\
 &\|I_{T,K_{\pm}}(n_-, u_+)(t)-I_{T,K_{\pm}}(m_-, v_+)(t)\|_{Y^s_{K_{\pm}}(I)} \notag \\
   &\le CT^{\th}(\|n_- - m_-\|_{\dot{Y}^s_{W_{-c}}(I)}\|u_+\|_{Y^s_{K_+}(I)} + \|m_-\|_{\dot{Y}^s_{W_{-c}}(I)}
              \|u_+ - v_+\|_{Y^s_{K_+}(I)}),  \label{np} \\
 &\|I_{T,K_{\pm}}(n_-, u_-)(t)-I_{T,K_{\pm}}(m_-, v_-)(t)\|_{Y^s_{K_{\pm}}(I)} \notag \\
   &\le CT^{\th}(\|n_- - m_-\|_{\dot{Y}^s_{W_{-c}}(I)}\|u_-\|_{Y^s_{K_-}(I)} + \|m_-\|_{\dot{Y}^s_{W_{-c}}(I)}
              \|u_- - v_-\|_{Y^s_{K_-}(I)}).  \label{nn}
}
Hence from $\|u_{\pm}\|_{Y^s_{K_{\pm}}(I)} \le r, \|m_{\pm}\|_{\dot{Y}^s_{W_{\pm c}}(I)} \le r$, 
\eqref{con}--\eqref{nn} and \eqref{T}, we have 
\EQS{
 &\|\Phi_1(u_{\pm}, n_+, n_-)-\Phi_1(v_{\pm}, m_+,m_-)\|_{Y^s_{K_{\pm}}(I)} \notag \\
  &\le (1/8)(\|u_+ - v_+\|_{Y^s_{K_+}(I)} + \|u_- - v_-\|_{Y^s_{K_-}(I)} \notag \\
   &\qquad + \|n_+ - m_+\|_{\dot{Y}^s_{W_{+c}}(I)} + \|n_- - m_-\|_{\dot{Y}^s_{W_{-c}}(I)}).            \label{con1}
}
Similarly, we have 
\EQS{
 &\|\Phi_2 (n_{\pm},u_+,u-) - \Phi_2 (m_{\pm}, v_+, v_-)\|_{\dot{Y}^s_{W_{\pm c}}(I)} \notag \\
  &=(4c)^{-1} (\|I_{T,W_{\pm c}}(u_+, u_+)(t)-I_{T,W_{\pm c}}(v_+, v_+)(t)\|_{\dot{Y}^s_{W_{\pm c}}(I)} \notag \\
   &\qquad \qquad + \|I_{T,W_{\pm c}}(u_+, u_-)(t)-I_{T,W_{\pm c}}(v_+, v_-)(t)\|_{\dot{Y}^s_{W_{\pm c}}(I)} \notag \\
   &\qquad \qquad + \|I_{T,W_{\pm c}}(u_-, u_+)(t)-I_{T,W_{\pm c}}(v_-, v_+)(t)\|_{\dot{Y}^s_{W_{\pm c}}(I)} \notag \\
   &\qquad \qquad + \|I_{T,W_{\pm c}}(u_-, u_-)(t)-I_{T,W_{\pm c}}(v_-, v_-)(t)\|_{\dot{Y}^s_{W_{\pm c}}(I)}).  \label{con2}
}
By Proposition \ref{BE}, we have 
\EQS{
 &\|I_{T,W_{\pm c}}(u_+, u_+)(t)-I_{T,W_{\pm c}}(v_+, v_+)(t)\|_{\dot{Y}^s_{W_{\pm c}}(I)}  \notag \\
  &\le CT^{\th}(\|u_+\|_{Y^s_{K_+}(I)}+\|v_+\|_{Y^s_{K_+}(I)})\|u_+ - v_+\|_{Y^s_{K_+}(I)}.                      \label{pp2}
}
Similarly, we have 
\EQS{
 &\|I_{T,W_{\pm c}}(u_+, u_-)(t)-I_{T,W_{\pm c}}(v_+, v_-)(t)\|_{\dot{Y}^s_{W_{\pm c}}(I)} \notag \\
  &\le CT^{\th}(\|u_+ - v_+\|_{Y^s_{K_+}(I)}\|u_-\|_{Y^s_{K_-}(I)} + \|v_+\|_{Y^s_{K_+}(I)}\|u_- - v_-\|_{Y^s_{K_-}(I)}),  
                                                                                                           \label{pn2} \\
 &\|I_{T,W_{\pm c}}(u_-, u_+)(t)-I_{T,W_{\pm c}}(v_-, v_+)(t)\|_{\dot{Y}^s_{W_{\pm c}}(I)} \notag \\
  &\le CT^{\th}(\|u_+ - v_+\|_{Y^s_{K_+}(I)}\|u_+\|_{Y^s_{K_+}(I)} + \|v_-\|_{Y^s_{K_-}(I)}\|u_- - v_-\|_{Y^s_{K_-}(I)}),         
                                                                                                           \label{np2} \\ 
 &\|I_{T,W_{\pm c}}(u_-, u_-)(t)-I_{T,W_{\pm c}}(v_-, v_-)(t)\|_{\dot{Y}^s_{W_{\pm c}}(I)}) \notag \\ 
  &\le CT^{\th}(\|u_-\|_{Y^s_{K_-}(I)}+\|v_-\|_{Y^s_{K_-}(I)})\|u_- - v_-\|_{Y^s_{K_+}(I)}. \label{nn2}
}
From $\|u_{\pm}\|_{Y^s_{K_{\pm}}(I)} \le r, \|v_{\pm}\|_{Y^s_{K_{\pm}}(I)} \le r$, \eqref{con2}--\eqref{nn2} and \eqref{T}, 
we obtain 
\EQS{
 &\|\Phi_2 (n_{\pm},u_+,u-) - \Phi_2 (m_{\pm}, v_+, v_-)\|_{\dot{Y}^s_{W_{\pm c}}(I)} \notag \\
  &\le (1/4)(\|u_+ - v_+\|_{Y^s_{K_+}(I)} + \|u_- - v_-\|_{Y^s_{K_-}(I)}).                        \label{con3}
}
Therefore, $(\Phi_1, \Phi_2)$ is a contraction mapping on $B_r(Y^s_{K_{\pm}}([0,T]) \cross \dot{Y}^s_{W_{\pm c}}([0,T]))$.  
Hence, by the Banach fixed point theorem, we have a solution to \eqref{int-eq} in it. 

Next, we prove uniqueness. 
Let $(u_{\pm}, n_{\pm}), (v_{\pm}, m_{\pm}) \in Y^s_{K_{\pm}}([0,T]) \cross \dot{Y}^s_{W_{\pm c}}([0,T])$ 
are two solutions satisfying $(u_{\pm}(0), n_{\pm}(0))=(v_{\pm}(0), m_{\pm}(0))$. 
Moreover, 
\EQQS{
 T' := \sup \{0 \le t \le T \, ; u_{\pm}(t)=v_{\pm}(t), n_{\pm}(t)=m_{\pm}(t) \} <T. 
}
By a translation in $t$, it suffices to consider $T'=0$. 
Fix $0 < \ta \le T$ sufficiently small. From \eqref{con}--\eqref{nn} and Proposition \ref{unique}, we obtain 
\EQS{
 &\|u_+ - v_+\|_{Y^s_{K_+}([0,\ta ])} \notag \\
 &\le (1/4)CT^{\th}\bigl{\{} (\|u_+\|_{Y^s_{K_+}([0,\ta ])} + \|u_-\|_{Y^s_{K_-}([0,\ta ])}) \notag \\
 &\qquad \cross  
                (\|n_+ - m_+\|_{\dot{Y}^s_{W_{+c}}([0, \ta ])} + \|n_- - m_-\|_{\dot{Y}^s_{W_{-c}}([0, \ta ])}) \notag \\
 &\qquad         + (\|m_+\|_{\dot{Y}^s_{K_+}([0,\ta ])} + \|m_-\|_{\dot{Y}^s_{K_-}([0,\ta ])}) 
          (\|u_+ - v_+\|_{Y^s_{K_+}([0, \ta ])} + \|u_- - v_-\|_{Y^s_{K_-}([0, \ta ])})\bigr{\}} \notag \\
 &\le (1/8)(\|n_+ - m_+\|_{\dot{Y}^s_{W_{+c}}([0, \ta ])} + \|n_- - m_-\|_{\dot{Y}^s_{W_{-c}}([0, \ta ])} \notag \\
 &\qquad          + \|u_+ - v_+\|_{Y^s_{K_+}([0, \ta ])} + \|u_- - v_-\|_{Y^s_{K_-}([0, \ta ])}).         \label{u-v}
}
From \eqref{u-v}, we obtain 
\EQS{ 
 &\|u_+ - v_+\|_{Y^s_{K_+}([0,\ta ])}  \notag \\
  &\le (1/7)(\|n_+ -m_+\|_{\dot{Y}^s_{W_{+c}}([0,\ta ])} + \|n_- - m_-\|_{\dot{Y}^s_{W_{-c}}([0,\ta ])} 
                + \|u_- - v_-\|_{Y^s_{K_-}([0,\ta ])}).      \label{upp}
}
Similarly, we have 
\EQS{
 &\|u_- - v_-\|_{Y^s_{K_-}([0,\ta ])}  \notag \\
  &\le (1/7)(\|n_+ -m_+\|_{\dot{Y}^s_{W_{+c}}([0,\ta ])} + \|n_- - m_-\|_{\dot{Y}^s_{W_{-c}}([0,\ta ])} 
                + \|u_+ - v_+\|_{Y^s_{K_+}([0,\ta ])}).      \label{unn}
}
From \eqref{con2}--\eqref{nn2} and Proposition \ref{unique}, we have 
\EQS{
 \|n_{\pm} - m_{\pm}\|_{\dot{Y}^s_{W_{\pm c}}([0,\ta ])} 
   \le (1/4)(\|u_+ - v_+\|_{Y^s_{K_+}([0,\ta ])} + \|u_- - v_-\|_{Y^s_{K_-}([0,\ta ])}).      \label{n-m}
}
Hence, collecting \eqref{u-v}--\eqref{n-m}, we obtain 
\EQQS{
 u_{\pm}=v_{\pm}, \quad n_{\pm} = m_{\pm}
}
on $[0,\ta ]$. 
This contradicts the definition of $T'$.

If $(u_{\pm 0}, n_{\pm 0}) \in B_{\de}(H^s(\R^d) \cross \dot{H}^s(\R^d))$ is radial, $s=s_c=d/2-2$ with $d \ge 4$ and 
$(u_{\pm}, n_{\pm}) \in B_r(Y^s_{K_{\pm}}(I) \cross \dot{Y}^s_{W_{\pm c}}(I))$ is spherically symmetric,  
then by Proposition \ref{BE}, we have 
\EQQS{
 &\|\Phi_1(u_{\pm}, n_+,n_-)\|_{Y^s_{K_{\pm}}(I)} \\ 
  &\le C\de + (1/4)C(\|n_+\|_{\dot{Y}^s_{W_{+c}}(I)}\|u_+\|_{Y^s_{K_+}(I)} 
         + \|n_+\|_{\dot{Y}^s_{W_{+c}}(I)}\|u_-\|_{Y^s_{K_-}(I)} \\
  &\qquad 
         + \|n_-\|_{\dot{Y}^s_{W_{-c}}(I)}\|u_+\|_{Y^s_{K_+}(I)} 
         + \|n_-\|_{\dot{Y}^s_{W_{-c}}(I)}\|u_-\|_{Y^s_{K_-}(I)}), \\
 &\|\Phi_2(n_{\pm},u_+,u_-)\|_{\dot{Y}^s_{W_{\pm c}}(I)} \\
  &\le C\de + (C/4c)(\|u_+\|^2_{Y^s_{K_+}(I)} + 2\|u_+\|_{Y^s_{K_+}(I)}\|u_-\|_{Y^s_{K_-}(I)} 
         + \|u_-\|_{Y^s_{K_-}(I)}^2). 
} 
Taking $\de = r^2$ and $r=\min\{1,c\}/(4C)$, then we have 
\EQQS{
 \|\Phi_1(u_{\pm}, n_+,n_-)\|_{Y^s_{K_{\pm}}(I)} \le r, \qquad 
 \|\Phi_2(n_{\pm},u_+,u_-)\|_{\dot{Y}^s_{W_{\pm c}}(I)} \le r. 
} 
Hence, $(\Phi_1, \Phi_2)$ is a map from $B_r(Y^s_{K_{\pm}}([0,T]) \cross \dot{Y}^s_{W_{\pm c}}([0,T]))$ into itself.  
If we also assume $(v_{\pm0}, m_{\pm 0}) \in B_{\de}(H^s(\R^d)) \cross \dot{H}^s(\R^d))$ is radial and  
$(v_{\pm}, m_{\pm}) \in B_r(Y^s_{K_{\pm}}(I) \cross \dot{Y}^s_{W_{\pm c}}(I))$ is spherically symmetric,  
then by the same manner as the estimate for \eqref{con1} and \eqref{con3}, we have 
\EQQS{
 &\|\Phi_1(u_{\pm}, n_+, n_-)-\Phi_1(v_{\pm}, m_+,m_-)\|_{Y^s_{K_{\pm}}(I)} \\
 &\le (1/8)(\|u_+ - v_+\|_{Y^s_{K_+}(I)} + \|u_- - v_-\|_{Y^s_{K_-}(I)} \\
   &\qquad + \|n_+ - m_+\|_{\dot{Y}^s_{W_{+c}}(I)} + \|n_- - m_-\|_{\dot{Y}^s_{W_{-c}}(I)}), \\
 &\|\Phi_2 (n_{\pm},u_+,u-) - \Phi_2 (m_{\pm}, v_+, v_-)\|_{\dot{Y}^s_{W_{\pm c}}(I)} \\
  &\le (1/4)(\|u_+ - v_+\|_{Y^s_{K_+}(I)} + \|u_- - v_-\|_{Y^s_{K_-}(I)}).
}
Thus, $(\Phi_1, \Phi_2)$ is a contraction mapping on $B_r(Y^s_{K_{\pm}}([0,T]) \cross \dot{Y}^s_{W_{\pm c}}([0,T]))$.  
Hence, by the Banach fixed point theorem, we have a solution to \eqref{int-eq} in it.   
We assume that $(u_{\pm}(0), n_{\pm}(0)), (v_{\pm}(0), m_{\pm}(0))$ are both radial and $s=s_c=d/2-2$ with $d \ge 4$.  
Let $(u_{\pm}, n_{\pm}), (v_{\pm}, m_{\pm}) \in Y^s_{K_{\pm}}([0,T]) \cross \dot{Y}^s_{W_{\pm c}}([0,T])$ 
are two spherically symmetric solutions satisfying $(u_{\pm}(0), n_{\pm}(0))=(v_{\pm}(0), m_{\pm}(0))$.
Then by the same manner as the proof for non-radial initial data, 
the uniqueness of the solution $(u_{\pm},n_{\pm})$ is showed.  
$(ii)$ follows from the standard argument, so we omit the proof. 
\end{proof}

Finally, we prove Proposition \ref{main_prop2}. The proof is the same manner as the proof for Proposition 4.2 in ~\cite{KaT}. 
\begin{proof}
There exists $M > 0$ such that for all $0 < T < \I$, 
\EQQS{
 &\|u_{\pm}\|_{Y^s_{K_{\pm}}([0,T])} + \|n_{\pm c}\|_{\dot{Y}^s_{W_{\pm c}}([0,T])} < M, \\
 &\|u_{\pm}\|_{Y^s_{K_{\pm}}([-T,0])} + \|n_{\pm c}\|_{\dot{Y}^s_{W_{\pm c}}([-T,0])} < M 
} 
holds since $r$ in the proof of Proposition \ref{main_prop1} does not depend on $T$. 
Take $\{t_k\}_{k=0}^K \in \mathcal{Z}_0$ and $0< T < \I$ such that $-T < t_0, t_K < T$. 
By $L^2_x$ orthogonality, 
\EQQS{
 &\Bigl( \sum_{k=1}^K \| \LR{\na_x}^s \bigl( K_{\pm}(-t_k)u_{\pm}(t_k) - 
        K_{\pm}(-t_{k-1})u_{\pm}(t_{k-1})\bigr)\|_{L^2_x}^2\Bigr)^{1/2} \\
 &\lec \| \LR{\na_x}^s u_{\pm}\|_{V^2_{K_{\pm}}([0,T])} + \| \LR{\na_x}^s u_{\pm}\|_{V^2_{K_{\pm}}([-T,0])} \\
 &\lec \|u_{\pm}\|_{Y^s_{K_{\pm}}([0,T])} + \|u_{\pm}\|_{Y^s_{K_{\pm}}([-T,0])} \\
 &< 2M.   
} 
Thus, 
\EQQS{
 \sup_{\{t_k\}_{k=0}^K \in \mathcal{Z}_0} \Bigl( \sum_{k=1}^K \| \LR{\na_x}^s K_{\pm}(-t_k)u_{\pm}(t_k) - 
        \LR{\na_x}^s K_{\pm}(-t_{k-1})u_{\pm}(t_{k-1}) \|_{L^2_x}^2\Bigr)^{1/2} 
 < 2M.
}
Hence, there exists $f_{\pm} := \lim_{t \to \pm \I} \LR{\na_x}^s K_{\pm}(-t)u_{\pm}(t)$ in $L^2_x(\R^d)$. 
Then put $u_{\pm} := \LR{\na_x}^{-s}f_{\pm}$, we obtain 
\EQQS{
 \| \LR{\na_x}^s K_{\pm}(-t)u_{\pm}(t)-f_{\pm}\|_{L^2_x} 
  = \| u_{\pm}(t) - K_{\pm}(t)u_{\pm \I}\|_{H^s_x} 
  \to 0  
} 
as $t \to \pm \I$.
The scattering result for the wave equation is obtained similarly.  

\end{proof}

\end{document}